\documentclass{amsart}

\def\mZ{\mathbb{Z}}
\usepackage{amsmath,amssymb,amsthm}
\usepackage{graphicx}

\usepackage{hyperref} 
\usepackage{cite} 

\newtheorem{theorem}{Theorem}
\newtheorem{lem}[theorem]{Lemma}
\newtheorem{proposition}[theorem]{Proposition}

\theoremstyle{definition}
\newtheorem{definition}[theorem]{Definition}

\theoremstyle{remark}
\newtheorem{remark}[theorem]{Remark}

\def\Cin{C_{\mathrm{inn}}}
\def\Cext{C_{\mathrm{ext}}}
\def\Bin{B_{\mathrm{inn}}}
\def\Bext{B_{\mathrm{ext}}}
\def\wh{\widehat}

\title{Canonical ordering for graphs on the cylinder, with applications
to periodic straight-line drawings on the flat cylinder and torus}

\author{Luca Castelli Aleardi$^{*}$ 
           \and Olivier Devillers$^{\dagger}$
        \and   \'Eric Fusy$^*$}

\thanks{\phantom{1}\hspace{-.6cm} $^*$ LIX - { \'Ecole Polytechnique}, Palaiseau, France, amturing,fusy@lix.polytechnique.fr.
Supported by the ANR grant ``EGOS'' 12-JS02-002-01
and the ANR grant ``GATO" ANR-16-CE40-0009-01.\\
$^{\dagger}$ INRIA  {  Nancy - Grand est}, France, olivier.devillers@inria.fr}

\begin{document}

\begin{abstract}
We extend the notion of canonical ordering (initially developed for planar triangulations and 3-connected planar maps) 
to cylindric (essentially simple) triangulations and more generally to cylindric (essentially internally) $3$-connected maps. 
This allows us to extend the incremental straight-line drawing algorithm of
 de Fraysseix, Pach and Pollack (in the triangulated case) and of Kant (in the $3$-connected case) to this setting. 
Precisely, for any cylindric essentially internally 
$3$-connected map $G$ with $n$ vertices, we can obtain in linear time a periodic (in $x$) 
 straight-line drawing of $G$ that is crossing-free and internally (weakly) convex,  
on a regular grid $\mZ/w\mZ\times[0..h]$, with $w\leq 2n$
and $h\leq n(2d+1)$, where $d$ is the face-distance between the two 
boundaries.  
This also yields an efficient periodic drawing algorithm 
for graphs on the torus. Precisely, for any essentially $3$-connected map $G$ on the torus (i.e., $3$-connected
in the periodic representation) with $n$ vertices,  
we can compute in linear time a periodic straight-line
drawing of $G$ that is crossing-free and (weakly) convex, on a periodic regular grid 
$\mZ/w\mZ\times\mZ/h\mZ$, with $w\leq 2n$ and $h\leq 1+2n(c+1)$, 
where $c$ is the face-width of $G$. Since $c\leq\sqrt{2n}$, the grid area is $O(n^{5/2})$. 
\end{abstract}

\maketitle

\pagestyle{plain}


\section{Introduction}

The problem of efficiently computing straight-line drawings of planar graphs has attracted a lot of  attention over the last two decades. 
Two combinatorial concepts for planar triangulations turn out to be the basis of many classical
straight-line drawing algorithms:
 the \emph{canonical ordering} (a special ordering of the vertices obtained by a shelling procedure)
and the closely related 
\emph{Schnyder wood} (a partition of the inner edges of a triangulation 
 into $3$ spanning trees with specific incidence conditions). 
Algorithms based on the canonical ordering~\cite{FPP90,Kan96,ChKa97,Miu01,Bra08} 
 are typically incremental, adding
vertices one by one while keeping the drawing planar.
Algorithms based on Schnyder woods~\cite{Sch90,Fe01,BoFeMo07} are more global, 
the (barycentric) coordinates of each vertex have a clear combinatorial meaning
(typically the number of faces in certain regions associated to the vertex). 
Algorithms of both types make it possible to draw in linear time a planar triangulation with $n$ vertices on a grid of size $O(n\times n)$. They can also both be extended~\cite{Fe01,Kan96} 
 to obtain (weakly) convex drawings of $3$-connected maps on a grid of size $O(n\times n)$. 
The problem of obtaining planar drawings of higher genus graphs has been addressed less
frequently~\cite{Ko01,Go06,Moh96,MR98,CEGL11,DGK11,Zit94}, 
from both the theoretical and algorithmic point of view.
Recently some methods for the straight-line planar drawing of genus $g$ graphs with polynomial
grid area (of size $O(n^3)$, in the worst case) have been described%
~\cite{CEGL11,DGK11}
(to apply these methods the graph needs to be unfolded planarly along a \emph{cut-graph}).  
However, these methods do not yield (at least easily)
periodic representations: for example, in the case of a torus, the  boundary vertices (on the boundary of the rectangular frame) might not be aligned, so that the drawing does not
give rise to a periodic drawing. A 
recent article~\cite{CFK14} achieves an adaptation 
of these methods to get alignment of opposite vertices
while keeping the size of the (periodic) grid polynomial, but 
with the drawback of having 
a quite large exponent, the guaranteed grid area beeing $O(n^8)$.  
Another method for drawing toroidal graphs 
with polynomial grid size is the algorithm of Gon\c{c}alves and L\'ev\^eque~\cite{GL11}, which is an adaptation of Schnyder's drawing principles to the torus\footnote{Their method relies
on the existence of certain orientations where
every vertex has degree $3$; they have also recently applied these
methods to design a bijective encoding scheme 
for toroidal triangulations~\cite{DGL15}; regarding 
extension to higher genus, it has been proved in~\cite{AGK16}
that a genus $g$ triangulation always admits an orientation
where every vertex outdegree is a non-zero multiple of $3$.}. It achieves both 
the periodicity requirement and polynomial grid-size; precisely the
size of the (periodic) regular grid is $O(n^2\times n^2)$ for simple
toroidal triangulations 
(no loops or multiple edges) 
and is $O(n^4\times n^4)$ for essentially
simple (simple in the periodic representation) toroidal
triangulations. 
Gon\c{c}alves and L\'ev\^eque 
also extend
these ideas to 3-connected toroidal maps (similarly, Schnyder woods for plane triangulations
have been extended to plane 3-connected maps~\cite{Fe01}), 
but as opposed to the planar case~\cite{Fe01},
the periodic drawings of 3-connected toroidal maps they obtain do not necessarily have the desired convexity property.

The main contributions of this article are  
efficient (in terms of grid size) algorithms to obtain crossing-free convex straight-line drawings
of (essentially) 3-connected maps on the cylinder and then on the torus. 
The key idea here is to adapt the principles of the iterative algorithms based
on canonical orderings~\cite{FPP90,Kan96} to the cylinder\footnote{Another  notion of canonical ordering for toroidal triangulations has
been introduced in~\cite{CFL09} (this actually works in any genus and yields
an efficient encoding procedure) but we will not use it here.}, first in the case of triangulations, then 3-connected maps
(as in the planar case, the 3-connected case is technically more involved).  
Precisely, we first adapt the notion of canonical ordering (a certain shelling procedure) 
and the incremental straight-line drawing algorithm of de Fraysseix, Pach and Pollack~\cite{FPP90}
(shortly called FPP algorithm thereafter) 
to triangulations on the cylinder (Section~\ref{sec:cylindric_drawing}). 
Then, more generally we can also extend the notion of canonical ordering and 
convex straight-line drawing of Kant~\cite{Kan96} to (essentially internally) 
$3$-connected maps on the cylinder (Section~\ref{sec:cyl3conn}).  
Precisely, for any essentially internally  
$3$-connected maps $G$ on the cylinder, our algorithm yields in linear time a 
 crossing-free internally convex straight-line drawing of $G$ 
 on a regular grid (on the flat cylinder) 
 of the form $\mZ/w\mZ\times [0..h]$, with $w\leq 2n$ and $h\leq n(2d+1)$, 
 where $n$ is the number of vertices of $G$
 and $d$ is the face-distance between the two boundaries of $G$ (smallest possible number of faces 
traversed by any curve connecting the two boundaries and meeting $G$ only at vertices).  

Then (in Section~\ref{sec:drawtorus}), we explain how to obtain periodic drawings on the torus 
by a reduction to the cylindric case (the reduction is done
with the help of a so-called \emph{tambourine}~\cite{Bon05}). 
For any essentially $3$-connected toroidal map $G$ (i.e.,
$3$-connected in the periodic representation) 
with $n$ vertices and so-called \emph{face-width} $c$ (smallest possible number of points of $G$  met by a non-contractible curve), 
we can compute in linear time a (weakly) convex periodic 
 straight-line drawing of $G$
 on a regular grid (on the flat torus) of size $w\times h$, with $w\leq 2n$ and $h\leq 1+2n(c+1)$. 
Since $c\leq(2n)^{1/2}$%
~\cite{Al78}, we have $h\leq(2n)^{3/2}$, 
so that the grid area is $O(n^{5/2})$. 
This improves upon the previously best known grid size for the torus, of $O(n^2\times n^2)$, by Gon\c{c}alves and L\'ev\^eque~\cite{GL11},
and  also always gives a (weakly)  convex drawing, which was not guaranteed in~\cite{GL11}. 

 
\vspace{.2cm}

\noindent{\bf Note.} This is the full paper version of a conference article that has appeared in the proceedings of the conference Graph Drawing'12, which only 
covered (with less details) the case of triangulations on the cylinder and on the torus. 

\section{Preliminaries}

\subsection{Graphs embedded on surfaces. }
A \emph{map} of genus $g$ is a connected 
graph $G$ embedded on the compact orientable 
surface $S$ of genus $g$, such that all components of $S\backslash G$ are
topological disks, which are called the \emph{faces} of the map. 
The map is called \emph{planar} 
for $g=0$ (embedding on the sphere) and \emph{toroidal} for $g=1$
(embedding on the torus). The \emph{dual} of a map $G$ is the map $G^*$ representing the adjacencies of
the faces of $G$, i.e., there is a vertex $v_f$ of $G^*$ in each face $f$ of $G$, and each edge $e$
of $G$ gives rise to an edge $e^*=\{v_f,v_{f'}\}$ in $G^*$, where 
$f$ and $f'$ are the faces on each side of $e$.  
A \emph{cylindric map} is a planar map $G$ with two marked faces $\Bin$ and $\Bext$ 
whose boundaries $\Cin$ and  $\Cext$ are simple cycles 
($\Cin$ and  $\Cext$ might share vertices and edges). The faces $\Bin$ and $\Bext$ are respectively called the 
\emph{inner boundary-face} and the \emph{outer boundary-face} (we will often consider cylindric maps  in the annular representation where $\Bext$ is the outer face,
as shown in Fig.~\ref{fig:cylindric_map} left-part).   
The other faces are called \emph{internal faces}.  
Boundary vertices and edges 
are those belonging to $\Cin$ 
(gray circles in Fig.~\ref{fig:cylindric_map}) 
or  $\Cext$ 
(black circles in Fig.~\ref{fig:cylindric_map}); 
 the other ones are called 
\emph{internal} vertices  
(white circles in Fig.~\ref{fig:cylindric_map})  
and edges. The notations $G$, $\Bin$, $\Bext$, $\Cin$, $\Cext$ will be used throughout the article. 

\begin{figure}[t]
\begin{center}
\includegraphics[width=10cm]{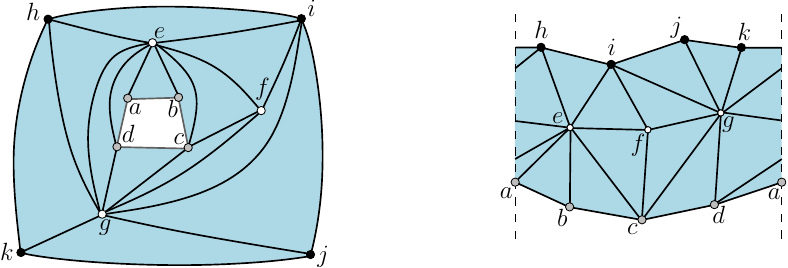}
\end{center}
 \caption{A cylindric triangulation with boundary faces $\Bin=\{a, b, c, d \}$ and $\Bext=\{h, i, j, k \}$. Left: annular representation. 
Right:  $x$-periodic representation.} 
 \label{fig:cylindric_map}
\end{figure}

\begin{figure}[t]
\begin{center}
\includegraphics[width=10cm]{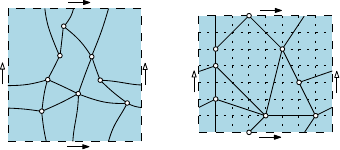}
\end{center}
 \caption{Left: an essentially 3-connected toroidal map $G$. Right: a (weakly) convex
straight-line drawing of $G$ on a periodic regular grid of size $8\times 7$.}
\label{fig:ex_torus_draw}
\end{figure}

\subsection{Periodic drawings.}
Here we consider the problem of drawing a cylindric map 
on the flat cylinder and drawing a toroidal map on the flat torus.  
For $w>0$ and $h>0$,  
the \emph{flat cylinder} of width $w$ and height $h$ is the 
rectangle $[0,w]\times[0,h]$ where the vertical sides
are identified. A point on this cylinder is located 
by two coordinates $x\in\mathbb{R}/w\mathbb{Z}$ and $y\in[0,h]$. 
The \emph{flat torus} of width $w$ and height $h$ is the rectangle $[0,w]\times[0,h]$
where both pairs of opposite
sides are identified. A point on this torus is located 
by two coordinates $x\in\mathbb{R}/w\mathbb{Z}$ and $y\in\mathbb{R}/h\mathbb{Z}$. 
Assume from now on that $w$ and $h$ are positive integers. 
For a cylindric map $G$, a \emph{periodic straight-line drawing}
of $G$ of width $w$ and height $h$ is a crossing-free straight-line drawing
(edges are drawn as segments, two edges can meet only at common end-points)
of $G$ on the flat cylinder of width $w$ and height $h$, such that the vertex-coordinates are in
$\mZ/w\mZ\times[0..h]$ (i.e., are integers).  
Similarly, for a toroidal map $G$, a \emph{periodic straight-line drawing}
of $G$ of width $w$ and height $h$ is a crossing-free straight-line drawing
(edges are drawn as segments, two edges can meet only at common end-points)
of $G$ on the flat torus of width $w$ and height $h$, such that the vertex-coordinates are in
$\mZ/w\mZ\times\mZ/h\mZ$ (i.e., are integers). 
A periodic straight-line drawing on the flat torus is said to be (weakly) \emph{convex}
if all corners have angle at most $\pi$, see Fig.~\ref{fig:ex_torus_draw} for an example.  
Note that a drawing of a toroidal triangulation
is automatically convex, so that convexity becomes a constraint only when there are faces
of degree larger than $3$. 

\section{Periodic drawings of cylindric triangulations}\label{sec:cylindric_drawing}
In this section we describe an algorithm to obtain periodic (in $x$) drawings of 
cylindric triangulations. These results are to be extended to 3-connected maps on the cylinder
in Section~\ref{sec:cyl3conn}; we start here with the case of triangulated maps for pedagogical reasons
(the different steps are the same as the ones to be used in the more general 3-connected case;
but the arguments at each step are simpler in the triangulated case). 

\subsection{Definitions and statement of the result}
A \emph{cylindric triangulation} is a cylindric map $T$  
such that all internal faces are triangles. A cylindric map is called
\emph{simple} if it has no loops nor multiple edges, and is called \emph{essentially simple}
if it has no loops nor multiple edges in the periodic representation. 
Note that an essentially simple cylindric map might have 2-cycles and 1-cycles (loops), which 
have to be \emph{non-contractible}
(they have $\Bin$ on one side and $\Bext$ on the other side), and two loops
can not be incident to a same vertex. 
We also define a \emph{chordal edge}, or \emph{chord}, at $\Cin$  
as an edge
not on $\Cin$ but with its two ends on $\Cin$. Similarly a \emph{chord}  
at $\Cext$ is an edge not on $\Cext$ but with its two ends on $\Cext$. 
For a cylindric map, the \emph{edge-distance} $d$ between the two boundaries is the length of a shortest possible
path starting from a vertex of $\Cin$ and ending at a vertex of $\Cext$ (possibly $d=0$). 
The main result obtained in this section is the following:

\begin{theorem}\label{thm:triang}
For each essentially simple cylindric triangulation $G$, one can compute in linear time
a crossing-free straight-line 
drawing of $G$ on an $x$-periodic regular grid $\mZ/w\mZ\times[0,h]$, where ---with $n$
the number of vertices and $d$ the edge-distance between the two boundaries---  
 $w\leq 2n$ and $h\leq 2n(d+1)$. 
In the drawing, the upper (resp. lower) boundary is a broken line monotone in $x$, formed by segments of slope at most $1$ in absolute value.
\end{theorem}

As a first step we will restrict to the case
with no chordal edge at $\Cin$:

\begin{proposition}\label{prop:triang_no_chord}
For each essentially simple cylindric triangulation $G$ with no chordal edge at $\Cin$, 
one can compute in linear time a crossing-free straight-line drawing of 
$G$ on an $x$-periodic regular grid $\mZ/w\mZ\times [0..h]$
where ---with $n$ the number of vertices of $G$ and $d$ the edge-distance between
the two boundaries--- $w\leq 2n$ and $h\leq n(2d+1)$, such that the upper boundary is a broken line monotone in $x$ formed by segments of slope in $\{+1,-1,0\}$
and the lower boundary is an horizontal line. 
\end{proposition}
To prove Proposition~\ref{prop:triang_no_chord} we will start with the subcase of cylindric simple triangulations.
In that case we will introduce a notion of canonical ordering (where it is necessary that $\Cin$ has no chordal edge), 
which makes it possible to design  an incremental periodic drawing algorithm, which can be seen as the cylindric counterpart of the FPP algorithm. 
Then we will extend the canonical ordering and periodic drawing algorithm to cylindric essentially simple triangulations
with no loop. We will then explain how to deal with (non-contractible) loops. This will establish Proposition~\ref{prop:triang_no_chord}.
Finally Proposition~\ref{prop:triang_no_chord} and handling chordal edges at $\Cin$ 
(as explained in Section~\ref{sec:chords_cin_triang}) will yield Theorem~\ref{thm:triang}.

\subsection{Canonical ordering for cylindric simple triangulations with no chord at $\Cin$.}\label{sec:canonical_triang} 
We introduce at first a notion of canonical ordering (classically studied on plane graphs) 
 for cylindric simple triangulations:

\begin{figure}
\begin{center}
\includegraphics[width=8cm]{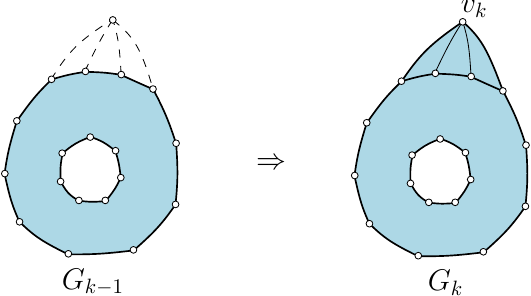}
\end{center}
\caption{From $G_{k-1}$ to $G_k$ in a cylindric canonical ordering (annular representation, only the boundaries and 
next added vertex and added edges are shown).}
\label{fig:one_step_triang}
\end{figure}

\begin{definition}\label{def:cylindricCanonicalOrderingTriang}
Let $G$ be a cylindric simple triangulation with no chordal edge at $\Cin$. 
An ordering $\pi= \{ v_1, v_2, \ldots, v_{n}\}$ of the vertices
of $G\backslash \Cin$ is called a \emph{(cylindric) canonical ordering}  
if it satisfies:
\begin{itemize}
\item
For each $k\in [0..n]$ the map $G_k$ induced by $\Cin$ and by the vertices $\{v_1,\ldots,v_k\}$
is a cylindric (simple) triangulation. The outer boundary-face of $G_k$ is denoted $C_k$. 
\item
For each $k\in [1..n]$, the vertex $v_k$ is on $C_k$, and its neighbours in $G_{k-1}$ are consecutive  on $C_{k-1}$
(see Fig.~\ref{fig:one_step_triang}).
\end{itemize}
\end{definition}
The notion of canonical ordering makes it possible to construct a cylindric
triangulation $G$ (with outer boundary $\Cext$ and inner boundary $\Cin$) 
incrementally, starting from $G_0=\Cin$ and 
adding one vertex at each step. 
This is similar to canonical orderings for planar triangulations, as introduced
by de Fraysseix, Pach and Pollack~\cite{FPP90} (the main difference is that, 
for a planar triangulation, one starts with $G_0$ being an edge,  
whereas here one starts with $G_0$ being
a cycle, seen as a cylindric map without internal faces).  

The computation of such an ordering is done by a shelling procedure similar to the one considered in the planar case~\cite{FPP90,Bre00}. 
At each step the graph formed by the remaining 
vertices  is a cylindric triangulation, the inner boundary remains $\Cin$
all the way, while the outer boundary (initially $\Cext$) has its contour, denoted 
 by $C_k$, 
getting closer to $\Cin$. 
A vertex $v\in C_k$ is \emph{free} if $v$ is
 incident to no chord of $C_k$ and if $v\notin \Cin$. 
 The shelling procedure goes as follows ($n$ is the number
 of vertices in $G\backslash \Cin$): ``for $k$ from $n$ to $1$,
 choose a free vertex $v$ on $C_k$, assign $v_k\leftarrow v$,
 and then delete $v$ together with all its incident edges''.
 The existence of a free vertex at each step follows from the same 
 argument as in the planar case~\cite{Bre00}. 
First, since there is no chord at $\Cin$, then as long as $C_k\neq \Cin$ there is at least one vertex on $C_k\backslash \Cin$. 
If there is no chord for $C_k$,
 then any vertex $v\in C_k\backslash \Cin$ is free.
If there is at least one chord $e=\{u,v\}$ for $C_k$, let $P_e$ be the 
path connecting $u$ and $v$ on $C_k$ such that the cycle $P_e+e$ does not enclose the inner boundary-face
(in the annular representation); and let $d_e$ be the length of $P_e$ (note that $d_e\geq 2$). 
Let $e=\{u,v\}$ be a chord such that $d_e$
is smallest possible. Then any vertex in $P_e\backslash\{u,v\}$ is free. Since
there exists a free vertex at each step, the procedure terminates. 
Let us now justify that the shelling procedure has linear time complexity. 
Note that an outer vertex $v\in C_k\backslash \Cin$ 
is free iff the number $N(v)$ of neighbours of $v$ on $C_k$ equals two.
So we just have to maintain $N(v)$ over all (current) outer vertices and put the free vertices (those for which $N(v)=2$)
in a stack, picking up the top element of the stack at each step. All this can be done in amortized time $O(|E|)$, with $|E|\leq 3n$ the number of edges
of the cylindric triangulation. To sum up:

\begin{figure}[t]
\begin{center}
\includegraphics[width=13cm]{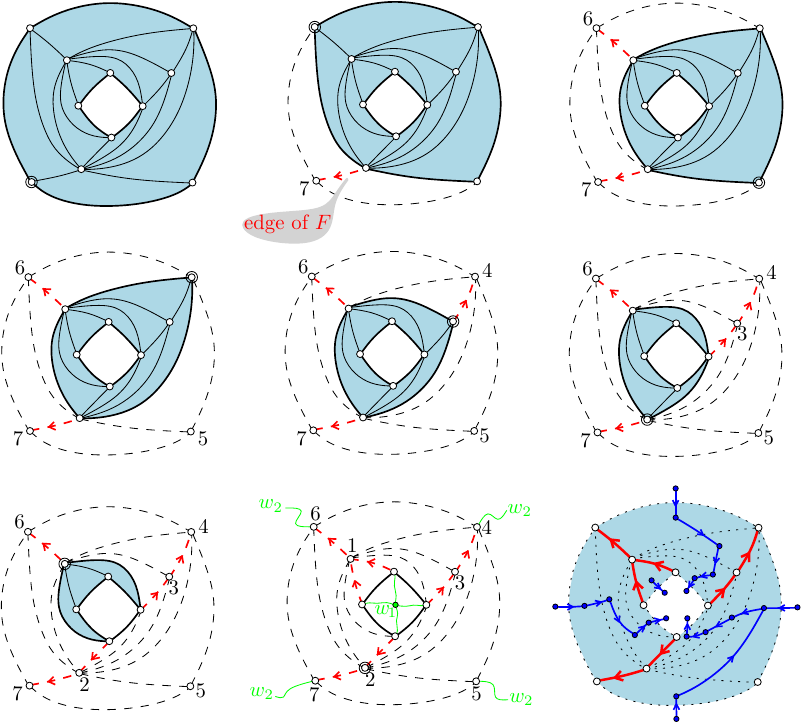}
\end{center}
\caption{Shelling procedure to compute a canonical ordering of a given 
cylindric triangulation (at each step the next shelled vertex is surrounded). 
The underlying forest is computed on the fly; the last drawing
shows the underlying forest superimposed with the dual forest.
The graph is the one of Fig.~\ref{fig:cylindric_map}.
}
\label{fig:shelling_triang}
\end{figure}

\begin{proposition}
Any cylindric simple triangulation $G$ with no chordal edge at $\Cin$ 
admits a (cylindric) canonical ordering that can be computed in linear time by a shelling procedure. 
\end{proposition}

\vspace{.4cm}

\noindent{\bf Underlying forest and dual forest.} 
Given a cylindric simple triangulation $G$ (with no chordal edge at $\Cin$) endowed with a canonical ordering $\pi$, we define  
the \emph{underlying forest} $F$ for $\pi$ as the oriented subgraph of $G$ where each vertex $v\in \Cext$  has outdegree $0$, and where each 
 $v\notin \Cext$ has exactly one outgoing edge, which is
connected to the adjacent vertex $u$ of $v$ of largest label in $\pi$.  
The forest $F$ can be computed on the fly during the shelling procedure: when treating 
an admissible vertex $v_k$, for each neighbour $v$ of $v_k$ such
that $v\notin C_k$, add the edge  $\{v,v_k\}$ to $F$, and orient it from $v$ to 
$v_k$.  
Since the edges are oriented in increasing labels, $F$ is an oriented forest; it spans all vertices of $G$ and has its roots on $\Cext$. 
The \emph{augmented map}  $\widehat{G}$ ($\widehat{G}$ has to be seen as a map on the sphere) is obtained from $G$ by adding a vertex 
$w_1$ inside $\Bin$, a vertex $w_2$ inside $\Bext$, and connecting 
 all vertices around $\Bin$ to $w_1$ and all vertices around $\Bext$ to $w_2$
 (thus triangulating the interiors of $\Bin$ and $\Bext$). 
 We denote by $\widehat{F}$ the forest $F$ 
plus all edges incident to $w_1$ and all edges incident to $w_2$.  
The \emph{dual forest} $F^*$ 
for $\pi$ is defined as the graph formed by the vertices of $\widehat{G}^*$ (the dual of  $\widehat{G}$)  and by the edges of $\widehat{G}^*$ 
that are dual to edges not in $\widehat{F}$. 
Since $\widehat{F}$ is a spanning connected subgraph of $\widehat{G}$, 
$F^*$ is a spanning forest of $\widehat{G}^*$. Precisely, each of the trees (connected components)
of $F^*$ is rooted at a vertex ``in front of'' each edge of $\Cin$, and the  
edges of the tree can be oriented toward this root-vertex,  
see Fig.~\ref{fig:shelling_triang} bottom right.
Each edge $e^*$ of $F^*$ is in a 
certain tree-component $T^*$ rooted at a vertex $v_0$ in front
of a certain edge of $\Cin$. Let $P$ be the path from $e^*$ to $v_0$ in $T^*$; $P$ is   
 shortly called \emph{the path from $e^*$ to the root}.  


\vspace{.4cm}

\subsection{Periodic drawing algorithm (simple triangulations, no chord at $\Cin$).}\label{sec:periodic_draw_triang} 
Given a cylindric simple triangulation $G$ without chord at $\Cin$,  
 we first  compute 
a canonical ordering of $G$, and then draw $G$ in an incremental way.   
We start with a cylinder of width $2|\Cin|$ and height $0$ (i.e.,
a circle of length  $2|\Cin|$) and draw the vertices of $\Cin$ equally spaced 
on the circle: space $2$ between two consecutive 
vertices\footnote{It 
is also possible to start with any configuration of points on a circle
such that any two consecutive vertices are at even distance}. 

Then the strategy for each $k\geq 1$ is to compute the drawing of $G_k$ out of the drawing of $G_{k-1}$.
Note that the set of vertices of $C_{k-1}$ that are neighbours of $v_k$
forms a path on $C_{k-1}$. Traversing this path $\gamma$ with the outer face 
of $G_{k-1}$  to the left, let $e_{\ell}$ be the first edge of $\gamma$
and $e_r$ be the last edge of $\gamma$ (note that $e_{\ell}=e_r$ if $v_k$ has only two neighbours on $C_{k-1}$). 
Let also $a_k$ be the starting vertex 
and let $b_k$ be the ending vertex of $\gamma$. 
Two cases can occur.

\vspace{.2cm}

 (1)~If, in the drawing of $G_{k-1}$ obtained
so far, $\mathrm{slope}(e_{\ell})<1$ and $\mathrm{slope}(e_r)>-1$, then we can directly insert $v_k$ in the drawing.
We place $v_k$ at the intersection of the ray of slope $1$ starting from $a_k$ 
and the ray of slope $-1$ starting from $b_k$, and we connect $v_k$ to all vertices
of $\gamma$ by segments. 

\vspace{.2cm}

(2)~If $\mathrm{slope}(e_{\ell})=1$ or $\mathrm{slope}(e_r)=-1$, then we can not directly insert $v_k$
as done in Case (1), because the edges $e_{\ell}$ and $\{a_k,v_k\}$ would overlap if $\mathrm{slope}(e_{\ell})=1$, or the edges
 $e_r$ and $\{b_k,v_k\}$ would overlap if $\mathrm{slope}(e_r)=-1$. We first have to perform stretching operations (thereby increasing the cylinder width by $2$) 
to make the slopes of $e_{\ell}$ and $e_r$ smaller than $1$ in absolute value. 
Define the \emph{$x$-span} of an edge $e$ in the cylindric drawing as the number of columns $[i,i+1]\times[0,+\infty]$ that meet the interior of $e$
(we have no need for a more complicated definition since, in our drawings, 
 a column will never meet an edge more than once).    
\begin{figure}[t]
\hspace*{-6mm}
\includegraphics[width=13.5cm]{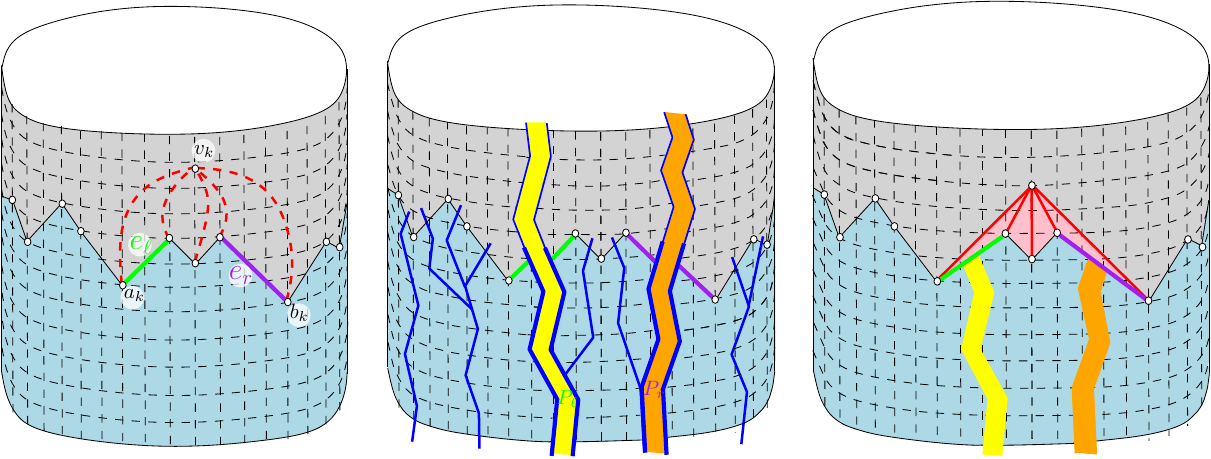}
\caption{One step of the incremental drawing algorithm. 
 Two vertical strips of width $1$ (each one along a path in the dual forest) 
are inserted in order to make 
the slopes of $e_{\ell}$ and $e_r$ smaller than $1$ in absolute value. Then the new vertex
and its edges connected to the upper boundary can be drawn in a planar way.}
\label{fig:one_step}
\end{figure}
Consider the dual forest $F^*$ for the canonical ordering restricted to $G_{k-1}$. 
Let $P_{\ell}$ (resp. $P_r$) be the path
in $F^*$ from $e_{\ell}^*$ (resp. $e_r^*$) to the root.  
We stretch the cylinder by inserting
 a vertical strip of length $1$ along $P_{\ell}$ and another along $P_r$, see Fig.~\ref{fig:one_step}.
 This comes down to increasing by $1$ the $x$-span of each edge of $G_{k-1}$ dual to an edge in $P_{\ell}$,
and then increasing by $1$ the $x$-span of each edge dual to an edge in $P_{r}$ 
 (note that $P_{\ell}$ and $P_r$ are not necessarily disjoint, in which case  
 the $x$-span of an edge dual to an edge in 
  $P_{\ell}\cap P_r$ is increased by $2$).  
%
 After these stretching 
operations,\footnote{In  
the FPP algorithm for planar triangulations, 
the step to make the (absolute value of) slopes of $e_{\ell}$ and $e_r$ smaller than $1$
is formulated as a shift of certain subgraphs described in terms of the underlying forest $F$.
The extension of this formulation to the cylinder would be quite cumbersome.
We find the alternative formulation with strip insertions more convenient for the cylinder. 
In addition it also gives rise to a very easy linear-time implementation (another linear-time implementation of the FPP algorithm
is given in~\cite{ChPa95}).} 
whose effect is to make the slopes of $e_{\ell}$ and $e_r$ strictly smaller than $1$ in absolute value, 
we insert, as in Case (1), the vertex $v_k$ at the intersection of the ray of slope $1$ starting from $a_k$ 
and the ray of slope $-1$ starting from $b_k$, and we connect $v_k$ to all vertices
of $\gamma$ by segments. 

\vspace{.2cm}

Note that in the two cases (1) and (2), the two rays from $a_k$ and $b_k$ actually intersect at a grid point since the Manhattan
distance between any two vertices on $C_{k-1}$ is even. 
 Fig.~\ref{fig:DrawingOnCylinder} shows the execution of the algorithm on the example of Fig.~\ref{fig:shelling_triang}. 


\begin{figure}[tp]
\hspace*{-4mm}
\includegraphics[width=13.18cm]{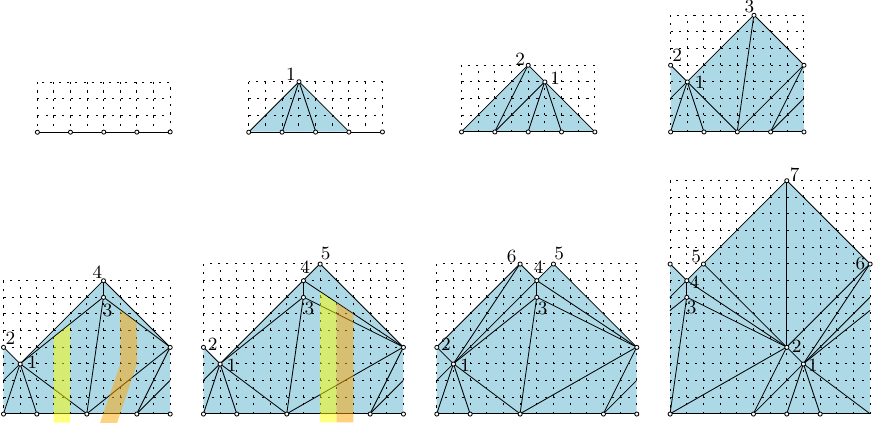}
 \caption{Complete execution of the algorithm computing an $x$-periodic drawing of a cylindric simple triangulation (no chordal edges incident to $\Bin$). The
vertices are treated in increasing label (the canonical ordering is the one computed
in Fig.~\ref{fig:shelling_triang}).}
\label{fig:DrawingOnCylinder}
\end{figure}

%
The fact that the drawing remains crossing-free relies on the fact that all edges of the upper boundary have slope at most $1$ in absolute value,
 and on the following inductive property (similar to the one used in~\cite{FPP90}), which is easily shown
to be maintained at each step $k$ from $1$ to $n$:

\vspace{.2cm}
 
\noindent\hspace{1cm}\begin{minipage}{11cm}
{\bf Pl}: for each edge $e$ on $C_k$ (the upper boundary of $G_k$), 
let $P_e$ be the path in $F^*$ from $e^*$ to the root, 
let $E_e$ be the set of edges dual to edges in $P_e$,  
and let $\delta_e$ be any nonnegative integer.  
Then the drawing remains planar after successively  increasing by $\delta_e$ the $x$-span of all edges of $E_e$, for all $e\in C_k$. 
\end{minipage}

\vspace{.2cm}

We now prove the bounds on the grid-size ($w$ is the width and $h$ the height of the cylinder on which $G$ is drawn).  
If $|\Cin|=t$  then the initial cylinder is $2t\times 0$; and at each vertex insertion, the grid-width grows by $0$ or $2$.
Hence $w\leq 2n$. In addition, due to the slope conditions (slopes of boundary-edges
are at most $1$ in absolute value), the $y$-span (vertical span) of every  edge $e$ 
is not larger than the current width  at the time when $e$ is inserted in the drawing. 
Hence, if we denote by $v$ the vertex of $\Cext$ that is closest (at distance $d$)
from $\Cin$, then the ordinate of $v$ is at most $d\cdot(2n)$.  
And due to the slope conditions, the vertical span of $\Cext$ in the drawing is at most 
$w/2\leq n$.  Hence the grid-height is at most $n(2d+1)$. The linear-time
complexity is shown next.

\vspace{.2cm}

\noindent{\bf Linear-time implementation.}  
An important remark is that, instead of computing the $x$-coordinates and $y$-coordinates
of vertices in the drawing, 
one can compute the $y$-coordinates of vertices and the $x$-span
of edges (as well as the knowledge of which extremity of the edge is the left-end vertex and which extremity is the right-end vertex).  
In a first pass (for $k$ from $1$ to $n$) 
one computes the $y$-coordinates of  vertices and the $x$-span $r_e$ 
of each edge $e\in G$ at the time $t=k$ when it appears on $G_k$ (as well one gets to know which extremity of $e$ is the left-end vertex). 
Afterwards if $e\notin F$,  
 the $x$-span of $e$ might further increase due 
to insertion of new vertices; denote by $s_e$ the total further increase undergone by $e$. 
Note that for each edge $e$ not in $F$,  if $e\notin\Cext$ 
there is a certain step $k$ such that $e\in C_{k-1}$ and $e\notin C_k$.    
Let $w_e\in\{0,1,2\}$ be defined as the stretch (increase of $x$-span) that $e$ undergoes just before adding $v_k$ to the drawing (in case $e\in\Cext$ no such step $k$ exists, and we assign $w_e=0$).
We call $w_e$ the \emph{weight} of $e$ (the quantities $w_e$ can be computed in a first pass, together with the quantities $r_e$). 
Let $P$ be the path from $e^*$ to the root. 
When stretching $e$ (before adding $v_k$), 
all edges dual to edges of $P$ undergo the same stretch 
(by $w_e$).  
In other words, if we denote by $T_e^*$ the subtree of $F^*$ hanging from $e^*$ (including $e^*$), and denote by $W_e$ the total weight of the (dual of the) edges in $T_e^*$, then $s_e=W_e$. 
Hence  the total $x$-span of each edge $e\in G$ is given by $r_e+s_e$,
where $s_e=0$ if $e\in F$ or $e\in \Cext$,
and $s_e=W_e$ if $e\notin F$ and $e\notin\Cext$. 
Since all quantities $s_e$ can easily be computed in linear time from the quantities $w_e$ (starting from the leaves and going
up to the roots of $F^*$),  
this gives a linear-time implementation. 

\vspace{.2cm}

To sum up, we have proved Proposition~\ref{prop:triang_no_chord} for simple cylindric triangulations with no chord at $\Cin$.

\vspace{.2cm}

\begin{remark}\label{rk:stretch_triang}
For each edge $e$ of $\Cin$,
let $r_e$ be the initial horizontal stretch of $e$ in the drawing 
procedure 
 (an even number, classically $r_e=2$ to have a compact drawing).
And let $t_e$ be the final horizontal stretch in the drawing procedure. The vectors $R=(r_e)_{e\in \Cin}$ and $T=(t_e)_{e\in \Cin}$
are called \emph{initial-stretch} and \emph{final-stretch} vectors relative to the drawing of $G$ (endowed with a given canonical ordering). 
Then the vector $S:=T-R$ is an invariant (does not depend on $R$), because $s_e=t_e-r_e$ just depends on the underlying forest (and dual forest) given by the canonical ordering.  
Hence, if with $R$ as initial stretch-vector we obtain a drawing with final-stretch vector $T$,  
then for any vector $T'$ of the form $T+2V$ ---with $V$ a vector of non-negative integers---
we can obtain a drawing with final-stretch vector $T'$ (by taking $R'=R+2V$ as initial-stretch vector instead of~$R$). 
\end{remark}

\begin{remark}\label{rk:recover_FPP}
Note that our algorithm can be seen as an extension of  the FPP algorithm, which works for simple planar quasi-triangulations (i.e., simple graphs embedded
in the plane with triangular inner faces and a polygonal outer face). If we are given a simple planar quasi-triangulation $Q$, we can
turn it into a cylindric simple triangulation $G$ by adding a vertex of degree $2$ connected to the two ends of the root-edge. 
Then the FPP drawing of $Q$ is recovered from the periodic drawing
of $Q$ upon deleting the added vertex, see Fig.~\ref{fig:FPP}. 
\end{remark}

\begin{figure}
\begin{center}
\includegraphics[width=12cm]{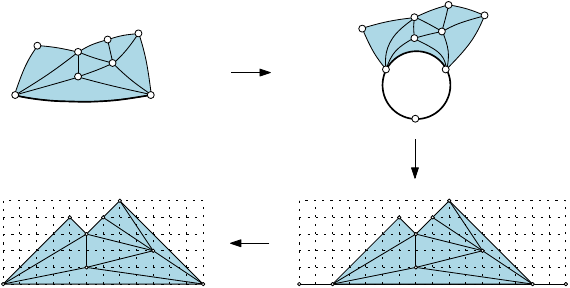}
\end{center}
\caption{The FPP algorithm for simple planar quasi-triangulations is recovered from our algorithm by adding a vertex of degree $2$ to complete
the inner boundary.}  
\label{fig:FPP}
\end{figure}

\subsection{Allowing for non-contractible 2-cycles (no chord at $\Cin$).}
The method (canonical ordering and incremental drawing algorithm) is easily extended to essentially simple cylindric triangulations
with no loop but possibly with non-contractible 2-cycles. Let $G$ be such a cylindric map with no chord at  $\Cin$. The definition of canonical ordering
for $G$ is exactly the same as for simple cylindric triangulations, adding the possibility that $G_k$ is obtained from $G_{k-1}$
as shown in Fig.~\ref{fig:step2}. 
\begin{figure}[t]
\begin{center}
\includegraphics[width=8cm]{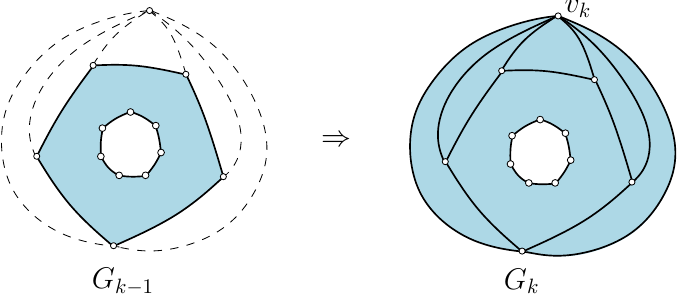}
\end{center}
\caption{When 2-cycles are allowed, the additional case shown here might occur
for the transition from $G_{k-1}$ to $G_k$.} 
\label{fig:step2}
\end{figure}
Such a canonical ordering can be computed by a shelling procedure that extends the one of Section~\ref{sec:canonical_triang}.
Call a $2$-cycle \emph{internal} if the two incident vertices are not both on the outer boundary.  
This time, a vertex on the outer boundary and not on the inner boundary is called \emph{free} if it is not incident to a chord nor incident to an internal $2$-cycle.

The shelling procedure consists in choosing a free vertex at each step, and deleting it together with its incident edges, until there just remains the inner boundary.
The existence of a free vertex, when the cylindric map is not reduced to a cycle, is proved as follows. 
First, since there is no chord at $\Cin$, there is at least one vertex on $\Cext\backslash \Cin$. 
If there is no chord nor internal 2-cycle incident to a vertex on $\Cext$, then any vertex on $\Cext\backslash \Cin$ is free. 
If there is at least one chord at $\Cext$, for each chord $e$ at $\Cext$ let $P_e$ be the path on $\Cext$ such that $P+e$ does not enclose $\Bin$, and let $d_e$ be the length of $P_e$. 
Let $e=\{u,v\}$ be a chord at $\Cext$ such that $d_e$ is smallest possible. Then any vertex $v\in P_e\backslash\{u,v\}$ is admissible. 
If there is no chord but there is at least one internal 2-cycle, consider the largest internal 2-cycle 
(in terms of containment; note that the 2-cycles are nested in the
annular representation). 
Then at least one outer vertex $v$ is strictly exterior to this $2$-cycle, hence is free (since we assume there is no chord). 

A linear time implementation is also readily obtained by maintaining, for each outer vertex, 
how many neighbours on $\Cext$ it has and how many internal 2-chords it is incident to. 
Note that such a canonical ordering also induces an underlying forest $F$ and an underlying dual forest $F^*$. These can be computed
on the fly during the shelling procedure, in the same way as for simple cylindric triangulations.  
Finally, the incremental drawing algorithm (and linear implementation using the dual forest) 
works exactly in the same way as for simple triangulations. An example is shown in Fig.~\ref{fig:shelling2}. 
The grid bounds are also the same as for simple triangulations (the arguments to obtain the bounds in the simple case did not use the fact
that there are no 2-cycles). So this gives Proposition~\ref{prop:triang_no_chord} for essentially simple triangulations with no loops.
Finally, note that Remark~\ref{rk:stretch_triang} still holds here (the arguments are the same).  

\begin{figure}
\begin{center}
\includegraphics[width=13cm]{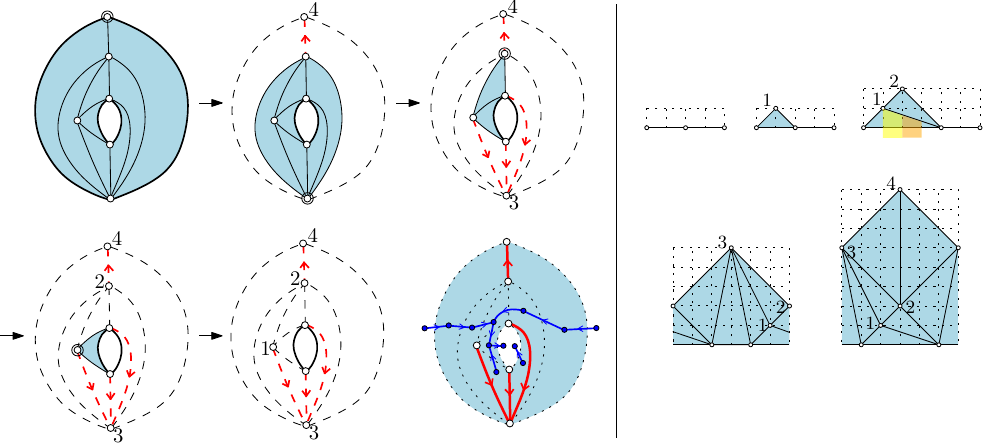}
\end{center}
\caption{Left-side: the shelling procedure for an essentially simple loopless cylindric triangulation $G$ 
with no chord at $\Cin$; the last drawing shows the underlying forest and dual forest.
Right-side: the incremental drawing algorithm.} 
\label{fig:shelling2}
\end{figure}

\vspace{.2cm}

\subsection{Allowing for non-contractible loops (no chord at $\Cin$).}
We finally explain how to deal with non-contractible loops. Our strategy is not to extend the notion of canonical ordering
but simply to decompose (at loops) such a cylindric map into a ``tower'' of components,
where the only loops in each component are at the boundary-faces. 
Let $G$ be an essentially simple cylindric triangulation with $n$ vertices and at least one loop. 
There are a few cases to treat:

\vspace{.2cm}

\emph{(a) $\Cin$ is the unique loop of $G$.} In that case, the algorithm of Section~\ref{sec:canonical_triang} (canonical ordering, shelling procedure, 
and incremental drawing procedure) works in the same way, see Fig.~\ref{fig:draw_one_inn_loop} for an example. Let $2m$ be the width of the drawing. 
By the arguments of Remark~\ref{rk:stretch_triang}, for any $m'\geq m$, $G$ has a periodic drawing of width $2m'$ and height at most $m'(2d+1)$, with $d$ 
the edge-distance between the two boundaries.

\begin{figure}
\begin{center}
\includegraphics[width=10cm]{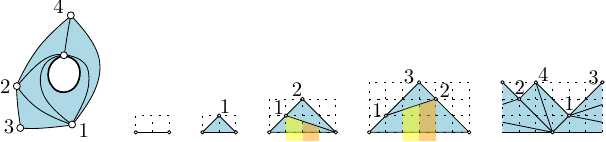}
\end{center}
\caption{Drawing algorithm when the inner boundary is the unique loop.} 
\label{fig:draw_one_inn_loop}
\end{figure}

\vspace{.2cm}

\emph{(b) $\Cext$ is a loop, $\Cin$ is possibly a loop, and there are no other loops.} Assume $\Cext$ is a loop and $G$ is not
reduced to that loop (i.e., $\Cext\neq \Cin$). Let $u$ be the vertex incident to the loop
(note that $u$ is not on $\Cin$ since there is no chord at $\Cin$), 
and let $c$ be the innermost 
2-cycle  incident to $u$. Cutting along $c$ (see Fig.~\ref{fig:draw_one_ext_loop}), we obtain two components: a planar triangulation $T$ and a cylindric
essentially simple triangulation $G'$ such that: $G'$ has no loop except possibly at $\Cin$, $G'$ has outer degree $2$ and $u$
is a free vertex for $G'$. Hence there is a canonical ordering for $G'$ such that $u$ is the first shelled vertex. 
Take a periodic drawing of $G'$ (for this canonical ordering) and take an FPP drawing of $T$. 
The widths of the respective drawings are even numbers, denoted $2n_1$ and $2n_2$. Let $m=\mathrm{max}(n_1,n_2)$. 
By the arguments of Remark~\ref{rk:stretch_triang} it is possible to redraw the graph that has the smaller grid-width so that it gets width $2m$, so that both drawings are of width $2m$. 
Then the drawing of $T$ (taken upside down) fits into the upper boundary of $G'$ yielding a periodic drawing of $G$ of width $2m$, see Fig.~\ref{fig:draw_one_ext_loop}.
The height of the drawing is at most $2dm\leq 2dn$, where $d$ is the edge-distance between the two boundaries (indeed, in the usual bound $h\leq (2d+1)m$, 
the $+1$ in the parenthesis is due to the vertical extension of the upper boundary, which is $0$ here). 

\begin{figure}
\begin{center}
\includegraphics[width=13cm]{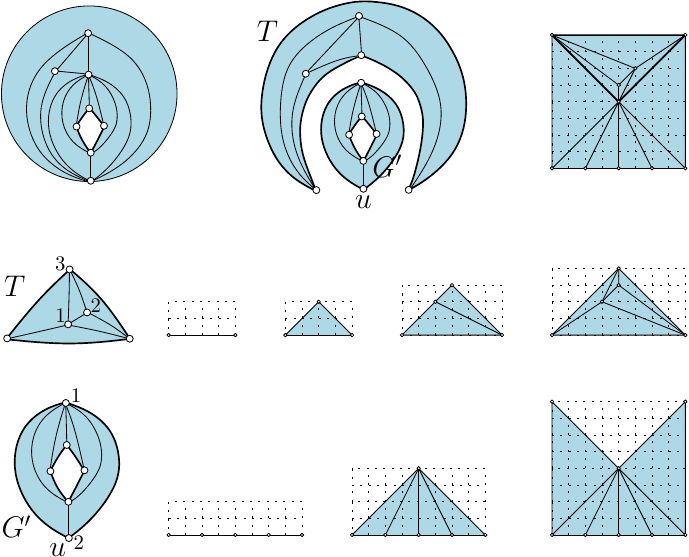}
\end{center}
\caption{Drawing algorithm when the outer face contour is a loop
(and there is no other loop except possibly at $\Cin$): 
the map is split along the innermost 2-cycle (the initial $x$-span of $T$ is taken to 
be $4$ instead of $2$ so that the drawings of $Q$ and $G'$ fit together).} 
\label{fig:draw_one_ext_loop}
\end{figure}

\vspace{.2cm}

\emph{(c) General case (without chord at $\Cin$).} We can assume there is at least one loop (the 
case with no loop has been covered in the last section). 
 Let $\ell_1,\ldots,\ell_r$ be the sequence of nested loops of $G$ (with $\ell_1$ the innermost loop
and $\ell_r$ the outermost loop) and let $G^{(0)},\ldots,G^{(r)}$
be the $r+1$ components that result from cutting successively along all these loops. 
For $i\in [0..r]$ let $d_i$ be the edge-distance between the two boundaries in $G^{(i)}$; and let $d$ be the edge-distance between the 
two boundaries of $G$. Note that $d=\sum_id_i$.  
Each component $G_i$ has  loops only at the 
 boundary-face contours, hence has a periodic drawing (according to cases (a) and (b)) 
such that boundaries that are loops 
are drawn as horizontal lines. Let $2n_1,\ldots,2n_r$ be the widths of the drawings of $G^{(1)},\ldots,G^{(r)}$ thus obtained. 
Let $m=\mathrm{max}(n_1,\ldots,n_r)$. By the arguments of Remark~\ref{rk:stretch_triang}, each of the graphs $G^{(i)}$ can be redrawn so as to have width $2m$. 
Stacking up all these drawings we obtain a periodic drawing of $G$ of width $2m$, see Fig.~\ref{fig:stack_loops}. 
Regarding the grid size, the width is $2m$, with clearly $m\leq n$, and the 
height of the drawing of $G_i$ is at most $2md_i$ for $i\in [0..r-1]$ and at most $m(2d_r+1)$ for $i=r$.  Hence the total height is at most $m(2d+1)\leq n(2d+1)$. 
This establishes Proposition~\ref{prop:triang_no_chord}.

\begin{figure}
\begin{center}
\includegraphics[width=12cm]{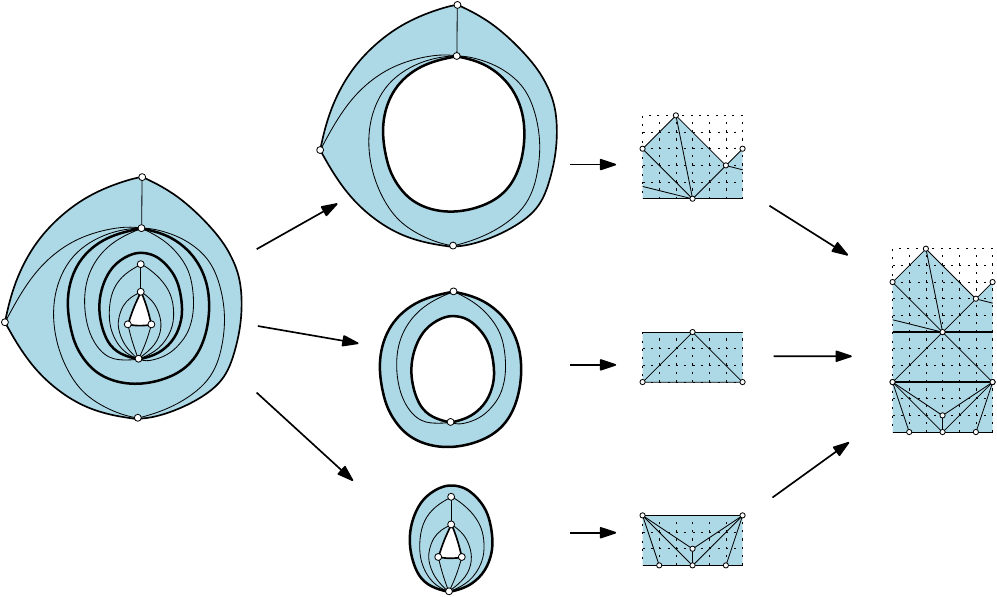}
\end{center}
\caption{Drawing an essentially simple triangulation $G$ (no chord at $\Cin$); $G$ is first decomposed
at its loops; each component is redrawn so that the component-drawings have the same width, and  
can be 
stacked up to obtain a periodic drawing of $G$.}
\label{fig:stack_loops}
\end{figure}

\vspace{.2cm}

\subsection{\bf Allowing for chords at $\mathbf{\Cin}$.}\label{sec:chords_cin_triang} 
We finally explain how to draw a cylindric essentially simple triangulation when 
  allowing for chords incident to $\Cin$.
It is good  to view $\Bext$ as the top boundary-face and $\Bin$ as
the bottom boundary face (and imagine a standing cylinder). 
For each chord  $e$ at the cycle $\Cin$, the \emph{component under $e$}, denoted $Q_e$,  
is the face-connected part of $G$ that lies below $e$; such 
a component is a quasi-triangulation (polygonal outer face, triangular inner faces)
rooted at the edge $e$.  
A chordal edge $e$ of $\Cin$ is \emph{maximal} if the component $Q_e$ under $e$
is not strictly included in the component under another chord at $\Cin$.
The \emph{FPP-size} $|e|$ of such an edge $e$ is defined as 
the width of the FPP drawing of $Q_e$.  
If we delete the component under each maximal chordal edge (i.e.,
delete everything from the component except for the chordal edge itself)
we get a new bottom cycle $C_0'$ that is chordless, so 
we can draw the reduced cylindric triangulation $G'$ using the algorithm
of Proposition~\ref{prop:triang_no_chord}. Let $w_e$ be the width 
of each edge $e$ of $C_0'$ in this drawing. 
According to Remark~\ref{rk:stretch_triang}, we can redraw $G'$ such that each edge $e\in C_0'$ that is chordal in $G$ 
has width $\ell(e)$, with $\ell(e)$ defined
as the smallest integer that is at least $\mathrm{max}(w_e,|e|)$ and such that $\ell(e)-w_e$ is even
(note that $\ell(e)\leq\mathrm{max}(w_e,|e|+1)$).  

\begin{figure}[t]
\begin{center}
\hspace{-.2cm}\includegraphics[width=12.4cm]{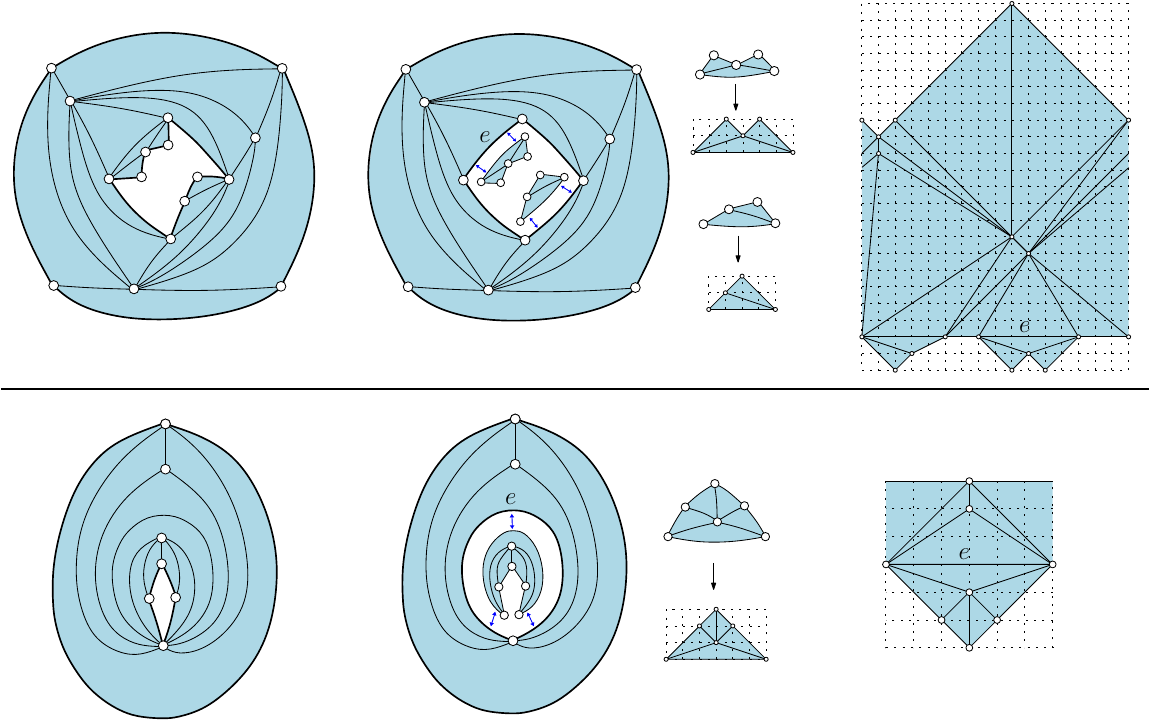}
\end{center}
 \caption{Drawing a cylindric triangulation with chords at $\Cin$ (the top-line example is simple, the 
bottom-line example is essentially simple with loops and 2-cycles in the annular representation). 
In the top-line example, to make enough space
to place the component under $e$, one takes $6$ (instead of $2$) as the initial $x$-span of $e$.
In the bottom-line example, to make enough space
to place the component under $e$, one takes $4$ (instead of $2$) as the initial $x$-span of $e$.
}
\label{fig:DrawingWithChords}
\end{figure}

Then for each maximal chord $e$ of $C_0$, we draw the component $Q_e$
under $e$ using the FPP algorithm. 
This drawing has width $|e|$, with  
$e$ as horizontal bottom edge of length $|e|$ and with the other
outer edges of slopes in $\pm 1$. We shift the left-extremity of $e$ to the left so that
the drawing of $Q_e$ 
gets width $\ell(e)$, then we rotate the drawing of $Q_e$ by $180$ degrees
and plug it into the drawing of $G'$, see Fig.~\ref{fig:DrawingWithChords}.      
The overall drawing of $G$ is clearly planar. 

We now give bounds on the grid-size of the overall drawing. 
Let $S$ be the sum of the FPP-sizes over all maximal chords $e$ at $\Cin$, and let $n'$ be the number of vertices of $G'$.
Clearly the width $w$ of the drawing of $G$ satisfies $w\leq 2n'+\sum_{e\in C_0'}\ell(e)-w_e\leq 2n'+S$.  
For each maximal chord $e$ at $\Cin$ let $n_e+2$ be the number of vertices of the component $Q_e$ under $e$.
Let $N$ be the sum of the quantities $n_e$ over all maximal chords at $\Cin$, so that $n=n'+N$. Since the FPP drawing of a quasi-triangulation with $p\geq 3$ vertices
has width at most $2p-4$, we have $|e|\leq 2n_e$ for each maximal chord at $\Cin$.  Hence $S\leq 2N$, so that $w\leq 2n$. 
Regarding the height of the drawing, by the same arguments as in Section~\ref{sec:periodic_draw_triang}, 
the height of the drawing of $G'$ is at most $n(2d+1)$, with $d$ the edge-distance
between the two boundaries. After adding the components under the chords at $\Cin$, the lower boundary is not horizontal anymore, but since
it is made of segments of slope at most $1$ in absolute value, it has vertical extension at
most $w/2\leq n$. Hence the overall height of the drawing is at most $n(2d+1)+n=2n(d+1)$. 
This finally yields 
the result pursued in this section, Theorem~\ref{thm:triang}.



\section{Periodic drawing of $3$-connected maps on the cylinder}\label{sec:cyl3conn}
We now extend the results obtained in Section~\ref{sec:cylindric_drawing} to the more
general 3-connected case. The approach is completely parallel to the one used in Section~\ref{sec:cylindric_drawing}, 
but the arguments at each step are more 
technical.\footnote{Similarly 
 in the planar case, the canonical ordering
and drawing algorithm in the 3-connected case, as developed by Kant~\cite{Kan96}, are more 
involved than the canonical ordering and drawing algorithm in the triangulated case, as developed
in the FPP algorithm~\cite{FPP90}.} 

\subsection{Definitions and statement of the result}

Let $G$ be a cylindric map, 
 embedded in the plane using the annular representation
(with $\Bext$ as the outer face and $\Bin$ as the 
marked inner face). 
We define a \emph{1-separating curve}
as a closed curve $\gamma$ not meeting any edge 
and intersecting $G$ in exactly one vertex,
 the unique visited face being an internal face,  
and such that the area enclosed by $\gamma$ contains at least one edge. Such a curve $\gamma$ is called \emph{non-contractible}  
of the area enclosed by $\gamma$ entirely contains $\Bin$. 
We now 
define a \emph{2-separating curve} as a closed curve
$\gamma$ intersecting $G$ in two vertices (not meeting any edge),
visiting exactly two faces, and such that the area enclosed by
 $\gamma$ strictly contains at least one vertex; again $\gamma$ is 
said to be \emph{non-contractible} if it entirely 
encloses $\Bin$. We also allow for the degenerate situation 
of a contractible 2-separating curve where the two
incident vertices are equal (it is still required
that the enclosed area strictly contains at least one vertex),  
see Fig.~\ref{fig:separating_curves} for an example.   
Two 1-separating curves (resp. two 2-separating curves)  are considered as equivalent if they are isotopic. It is convenient ---and will
always 
be assumed from now on--- to discard non-contractible 
 curves passing by the outer face. 

\begin{figure}
\begin{center}
\includegraphics[width=8cm]{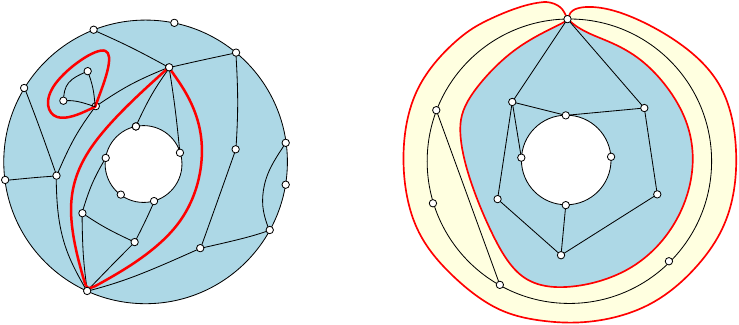}
\end{center}
\caption{Left: a cylindric map $G$ with a (contractible) 1-separating curve and a
 (non-contractible and nondegenerate) 2-separating curve (in bold line).
 Right: a cylindric map with a contractible degenerate 2-separating
curve (the enclosed area appears in yellow). 
\label{fig:separating_curves}}
\end{figure}

We consider in this section certain cylindric maps $G$ with some marked vertices on $\Cext$ and some marked vertices on $\Cin$ which are called \emph{active} 
(if a vertex is on $\Cin\cap \Cext$ it might be active for $\Cin$ or $\Cext$, or both, or none). As we will see, when all 
active vertices are at $\Cext$,   
these vertices are 
to be the ones that are allowed to be selected during 
the shelling procedures to compute a canonical ordering
in the more general 3-connected case presented here. 
  The terminology of active vertices will also be useful
when presenting the drawing algorithm for 
toroidal 3-connected maps via a reduction to the cylindric case.


A cylindric map $G$ with some active vertices is called \emph{internally $3$-connected} if is has  no 1-separating curve, 
and any 2-separating curve is contractible (and nondegenerate) 
and strictly 
encloses at least one active vertex. 
And $G$ is called \emph{essentially internally $3$-connected} if there is   
no contractible 1-separating curve,  and 
any contractible 2-separating curve (possibly degenerate) strictly encloses  at least one active vertex. 
Being essentially internally $3$-connected can also be conveniently
characterized in the $x$-periodic representation $\hat{G}$ of $G$:
a \emph{1-separating} (resp. \emph{2-separating}) 
curve in $\hat{G}$ is a simple closed curve not meeting
 any edge of $\hat{G}$, 
meeting $\hat{G}$ at exactly 1 (resp. 2) vertex, and whose interior
strictly contains at least one vertex. Then $G$ is 
essentially internally 3-connected iff, in $\hat{G}$, there
is no 1-separating curve and any 2-separating curve 
strictly encloses at least one active vertex.

For a cylindric map $G$ with active vertices, a periodic straight-line drawing of $G$ is called \emph{convex} if all corners have angle at most $\pi$, except possibly for 
corners of $\Bin$ (resp. $\Bext$) at an active vertex for $\Cin$ (resp. for $\Cext$), whose
angle in the drawing can be larger than $\pi$.    
For a cylindric map $G$, the \emph{face-distance} between the two boundaries is the minimal possible integer $q$ such that there is a 
 curve in $G$ starting from a vertex of $\Cin$, ending at a vertex of $\Cext$, not 
meeting any edge, 
and passing by $q$ (internal) faces of $G$. 
The main result obtained in this section is the following:

\begin{theorem}\label{thm:triconn_cyl}
For each essentially internally $3$-connected cylindric map $G$ with at least one active vertex, one can compute in linear time
a periodic convex 
drawing of $G$ on an $x$-periodic regular grid $\mZ/w\mZ\times[0,h]$, where ---with $n$
the number of vertices and $d$ the face-distance between the two boundaries---  
 $w\leq 2n$ and $h\leq 2n(d+1)$. 
In the drawing, the upper (resp. lower) boundary is a broken line monotone in $x$ formed by segments of slope at most $1$ in absolute value.
\end{theorem}

As a first step we will prove the result
when there is no active vertex on $\Cin$:

\begin{proposition}\label{prop:triconn_cyl_first}
For each essentially internally $3$-connected cylindric map $G$ with at least one active vertex on $\Cext$ and no active vertex on $\Cin$, 
one can compute in linear time a periodic convex 
drawing of $G$ on an $x$-periodic regular grid $\mZ/w\mZ\times[0,h]$, where ---with $n$
the number of vertices and $d$ the face-distance between the two boundaries---  
 $w\leq 2n$ and $h\leq n(2d+1)$. 
In the drawing, the upper boundary is a broken line, monotone in $x$, formed by segments of slope at most $1$ in absolute value, and the lower
boundary is an horizontal line.
\end{proposition}

Note that Theorem~\ref{thm:triconn_cyl} and Proposition~\ref{prop:triconn_cyl_first} are respectively extensions of Theorem~\ref{thm:triang} and 
Proposition~\ref{prop:triang_no_chord}.   
Indeed, for an essentially simple cylindric triangulation $G$, making active all boundary-vertices of $G$ yields an essentially internally $3$-connected cylindric map,
and if $G$ has no chord at $\Cin$, then making active all vertices of the outer boundary yields an essentially internally $3$-connected cylindric map with no 
active vertex at $\Cin$. In addition, for an essentially simple cylindric triangulation, the face-distance between the two boundaries 
coincides with the edge-distance between the two boundaries. 
To prove Proposition~\ref{prop:triconn_cyl_first} (the strategy is parallel to the one we have followed
to prove Proposition~\ref{prop:triang_no_chord} for cylindric triangulations) 
we will start with the subcase where $G$ is internally $3$-connected. 
In that case we will introduce a notion of canonical ordering (which extends both the canonical ordering for cylindric  
triangulations introduced in Section~\ref{sec:canonical_triang}, and the canonical ordering for internally $3$-connected plane graphs introduced
by Kant~\cite{Kan96}).    
This makes it possible to design  an incremental periodic drawing algorithm, which is the cylindric counterpart (and extension) of the algorithm
introduced by Kant~\cite{Kan96} in the planar case  (which itself extends the FPP algorithm to $3$-connected planar maps).
Then we will extend the canonical ordering and drawing algorithm to the subcase where there is no $1$-separating curve;     
after which we will explain how to deal with $1$-separating curves. This will establish Proposition~\ref{prop:triconn_cyl_first}.
We will then explain how to deal with active vertices at $\Cin$. This will yield Theorem~\ref{thm:triconn_cyl}.

\subsection{Restatement of the definitions in terms of the corner-map}
We provide here a classical reformulation of the 3-connectedness conditions in terms of the 
so-called corner-map, which provides a more combinatorial way of viewing 2-separating curves.  
Given a cylindric map $G$ (whose vertices are considered as white), in its annular representation,
 the \emph{corner-map} $S$ of $G$ is obtained by inserting a black vertex $v_f$ 
in each internal face $f$ 
of $G$ and connecting $v_f$ to all the corners around $f$; $S$ is the graph made of black and white
vertices and of the (newly added) edges between black and white vertices. The \emph{completed map} $\widehat{G}$ of $G$ is defined as $G$ superimposed with $S$, see Fig.~\ref{fig:corner_map}
for an example.  
A \emph{separating 4-cycle} in $S$ is a 4-cycle containing 
at least one vertex in its interior; it is \emph{contractible}
if it encloses $\Bin$.  

\begin{figure}
\begin{center}
\includegraphics[width=13cm]{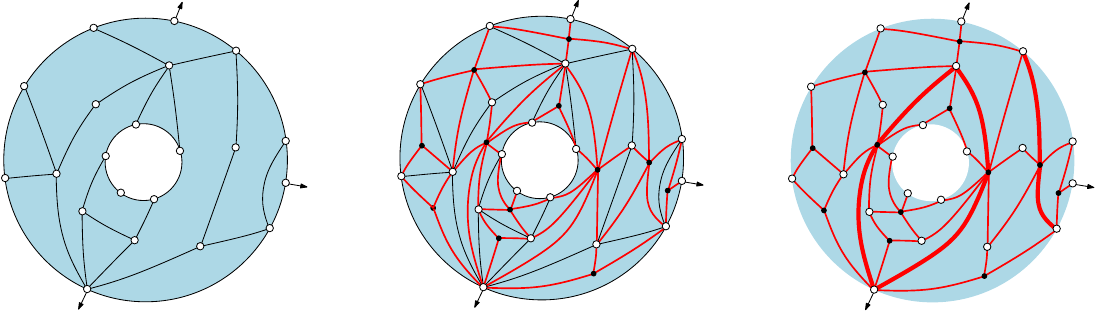}
\end{center}
\caption{Left: an essentially internally 3-connected cylindric map $G$. 
Right: the associated corner-map $S$ 
(in bolder line, 
a (non-contractible) separating 4-cycle, which corresponds to a 2-separating curve of $G$ passing by two internal faces,
and a (contractible) separating 2-chord. Middle: the completion map $\widetilde{G}$, which is obtained by superimposing $G$ and~$S$.}
\label{fig:corner_map}
\end{figure}

We define (in $\wh{G}$) a \emph{2-chord} $\gamma$ for $\Cext$ 
as a path $e_1,e_2$ of length two in $S$ starting from a 
vertex $u$ of $\Cext$ and ending at a vertex $v$ of $\Cext$. We denote by $P_{\gamma}$ the path from $u$ to $v$ on $\Cext$ such that the cycle $C_{\gamma}:=P_{\gamma}\cup\gamma$
does not contain $\Bin$ (if $u=v$ then $P_{\gamma}$ is taken as the whole contour of $\Cext$); $C_{\gamma}$ is called the \emph{cycle enclosed by the 2-chord}. 
The 2-chord $\gamma$ is called \emph{separating}
if $P_{\gamma}$ is of length larger than $1$, i.e., has at least one internal vertex. The internal
vertices of $P_{\gamma}$ are said to be \emph{enclosed} by $\gamma$. 
We also allow for a \emph{degenerate} situation where $u=v$
(corresponding to a degenerate 2-separating curve passing 
by the outer face), 
in which case $P_{\gamma}$ is taken as the whole contour of $\Cext$. 

Similarly one
defines a \emph{2-chord} for $\Cin$ (and the related notions). 
We now characterize  the conditions of being
internally (resp. essentially internally) 3-connected in terms of the corner-map and completed map. 
 In terms
of the corner-map, a cylindric map $G$ where $\Cext$ and $\Cin$  do not meet 
is internally 3-connected iff $S$ has no 2-cycle
nor separating 4-cycle, and any separating 2-chord $\gamma$ at $\Cext$ (resp. at $\Cin$)  is nondegenerate and 
encloses at least one active vertex for $\Cext$ (resp. for $\Cin$). 
If $\Cext$ meets $\Cin$, we define an intersection-vertex as a vertex of $\Cext\cap\Cin$; 
$G$ can be seen as a cyclic sequence of elementary blocks  which
are attached along the intersection-vertices. Such elementary blocks are called the \emph{portions}
 of $G$, and the two intersection-vertices delimiting the portion are called the \emph{extremal vertices} of the portion. A portion is said to be \emph{non-trivial} if it is not reduced to an edge. 
If $\Cext$ meets $\Cin$, then $G$ has to satisfy the further property that each non-trivial
portion has at least one non-extremal vertex that is active (either for $\Cext$ or for $\Cin$). 

We now state the conditions for ``essentially internally 3-connected''. 
When $\Cext$ does not meet $\Cin$, a cylindric map $G$
is essentially internally 3-connected iff any 2-cycle or separating 4-cycle of $S$ encloses $\Bin$,
no such 2-cycles can meet, and any separating  2-chord $\gamma$ at $\Cext$ (resp. at $\Cin$)   
encloses at least one active vertex for $\Cext$ (resp. for $\Cin$). And when $\Cext$ meets $\Cin$
one can similarly consider the \emph{portions} of $G$, and one has the further requirement
that any non-trivial portion has at least one non-extremal vertex that is active (either for $\Cext$ or for $\Cin$).

\subsection{Canonical ordering}\label{sec:canonical_3conn} 
We first introduce a notion of canonical ordering 
 for internally $3$-connected cylindric maps. Before that, we need a bit of terminology. Given a pair $(u,w)$ of outer active vertices,
 the \emph{outer path for $(u,w)$} is the path, denoted $\gamma(u,w)$, from $u$ to $w$ on the outer face contour, having the outer face on its left. 
The pair $(u,w)$ is called \emph{consecutive} if there is no active vertex in $\gamma\backslash\{u,w\}$. 
\begin{definition}\label{def:cylindricCanonicalOrdering3conn}
Let $G$ be an internally $3$-connected cylindric map with no active vertex for $\Cin$ and with at least one active vertex for $\Cext$.   
A \emph{canonical ordering} for $G$ is a growing sequence $G_0, G_1,\ldots,G_p$ of internally $3$-connected cylindric maps such that:
\begin{itemize}
\item
Initially, $G_0=\Cin$, at the end $G_p=G$. All along, the inner boundary-face of $G_k$ is $\Bin$. The outer boundary 
of $G_k$ is denoted $C_k$. 
\item
For $k\in[0..p]$, the active vertices on $C_k$ are those with at least
one neighbour in $G\backslash G_k$ and those that are on $\Cext$ 
and are active (for $G$).\footnote{With  
these conditions it is easy to check that, since  
$G$ has at least one active vertex for $\Cext$, then $G_k$ must have at least one active vertex for 
$C_k$.}
\item
For $k\in[1..p]$, $G_k$ is obtained from $G_{k-1}$ either by:
\begin{itemize}
\item
choosing a non-consecutive pair $(u,w)$ of distinct active vertices on $C_{k-1}$ 
and connecting all active vertices on $\gamma(u,w)$ to a newly added vertex in the outer face (see Fig.~\ref{fig:step_canonical_3conn}(a)); 
\item
or choosing a consecutive pair $(u,w)$ of distinct active vertices on $C_{k-1}$ and adding a path of at least two edges in the outer face connecting these two vertices 
(see Fig.~\ref{fig:step_canonical_3conn}(b)).
\end{itemize}
\item
Additionally, for $k\in[1..p]$, the vertices of $G_k\backslash G_{k-1}$ must be active (for $G_k$).
\end{itemize}
\end{definition}

\begin{figure}[t]
\begin{center}
\includegraphics[width=12cm]{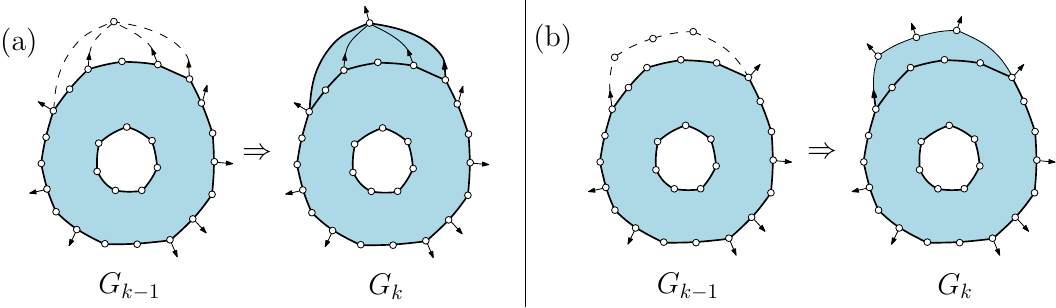}
\end{center}
 \caption{The two possible transitions from $G_{k-1}$ to $G_k$ in a  canonical ordering of 
an internally $3$-connected cylindric map.}
\label{fig:step_canonical_3conn}
\end{figure}

\begin{figure}[t]
\begin{center}
\includegraphics[width=12cm]{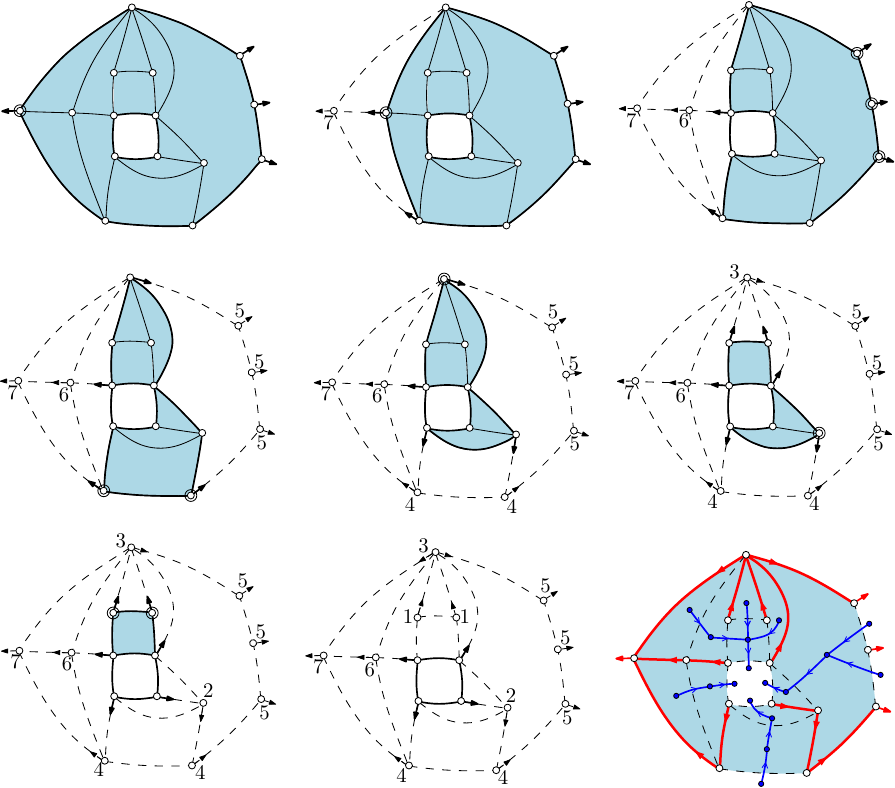}
\end{center}
 \caption{
Shelling procedure to compute a canonical ordering of an internally $3$-connected cylindric map (at each step the next deleted vertices are surrounded,
the ranks of successively shelled vertices are also indicated).  
The last drawing shows the cylindric map with the underlying forest and dual forest.}
\label{fig:canonical_3conn_complete}
\end{figure}

Given a canonical ordering of $G$, the \emph{rank} of a vertex $v\notin \Cin$ is the smallest $k$ such that $v\in G_k$. 
We prove the existence of such a canonical ordering by induction on the number of internal faces, which also yields
a shelling procedure similar to the one described in the triangulated case.  
A first remark is that if $G$ is not reduced to a cycle, then $G$ must have an active vertex for $\Cext$
that is not on $\Cext\backslash \Cin$ (if all the active 
vertices of $G$ were on $\Cext\cap\Cin$, it would easily yield 
a $2$-separating curve not enclosing any active vertex, a contradiction). 
To describe the shelling procedure we need a bit of terminology. 
An internal face $f$ is called \emph{separating}
if there is a separating $2$-chord (always for $\Cext$ here) passing by $f$. 
A vertex on the outer face is called \emph{admissible}
if it is active, not on $\Cin$, and not incident to any separating face.  
Note that, by deleting an admissible vertex and declaring all its neighbours as active, one gets a cylindric map $G'$ such that $G$ is obtained from $G'$
by applying the operation of Fig.~\ref{fig:step_canonical_3conn}(a). 
We now prove that either there is an admissible vertex or it is possible to get $G$ from a cylindric map $G'$ with one less internal face
by applying the operation of Fig.~\ref{fig:step_canonical_3conn}(b).  
For a given separating face $f$,  
the \emph{maximal separating $2$-chord for $f$}
is the separating $2$-chord $c$ passing by $f$ and such that the cycle enclosed by $c$
 is maximal (for the containment relation). 
Consider the set $E$ of maximal separating 
$2$-chords associated to all separating internal faces of $G$. 
If $E$ is not empty (i.e., there is at least one separating internal face), let $c$ be a separating 2-chord in $E$ that is minimal (for the containment relation of enclosed cycles) 
and let $f$ be the internal face to which $c$ belongs.  
Two cases can occur. If the cycle enclosed by $c$ contains no other internal face than $f$, 
then the 
subgraph $G_c$ of $G$ in the cycle enclosed by $c$ (including the boundary of the cycle) 
is a path $P$ that has at least two edges. Then by deleting the edges and internal vertices of $P$, 
and declaring as active the two extremities of $P$,   
we obtain a cylindric map $G'$ such that $G$ is obtained from $G'$
by applying the operation of Fig.~\ref{fig:step_canonical_3conn}(b). Otherwise $G_c$ 
contains at least one internal face $f'$. 
Among the separating $2$-chords passing by $f$ and whose enclosed cycle
contains $f'$, let $c'$ be the minimal one (for the containment relation
of enclosed cycles). Let $P$ be the path on $\Cext$
inside $c'$, let $u,w$ be the extremities of $P$. 
By minimality of $c'$, all vertices of $P\backslash\{u,w\}$ are not incident to $f$. In addition, at least one of these vertices  
has to be active (since $G$ is internally $3$-connected). Let $v$ be such a vertex. 
By minimality of $c$, 
all internal faces (except for $f$) in the cycle enclosed by $c$ are non-separating, hence $v$
 (which is not incident to $f$) is admissible. 
In case $E$ is empty, then there is no separating face. As we have seen, there is at least one active vertex on $\Cext\backslash \Cin$.
Since there is no separating face, this vertex is admissible. 
To sum up, in all cases, it is possible to obtain $G$ from a smaller cylindric map $G'$ by applying the operation of Fig.~\ref{fig:step_canonical_3conn}(a) 
or Fig.~\ref{fig:step_canonical_3conn}(b). 
It is also readily checked that the outer face of $G'$ is a simple cycle and that $G'$ satisfies the conditions of Definition~\ref{def:cylindricCanonicalOrdering3conn}. 
So we can continue inductively (starting from $G'$) until there just remains the cycle $\Cin$, 
which yields a shelling procedure
to compute a canonical ordering satisfying Definition~\ref{def:cylindricCanonicalOrdering3conn}. An example is shown in Fig.~\ref{fig:canonical_3conn_complete}.  
 
Let us now justify that the shelling procedure has linear time complexity. 
At each step, for each internal face $f$ whose contour meets the current outer boundary $C_k$, 
we denote by $V(f)$ the number of outer vertices incident to $f$ and by $E(f)$
the number of outer edges on the contour of $f$.
Similarly as in Kant~\cite{Kan96} for the planar case, we note that, at each step, an internal face is non-separating iff $E(f)\leq 1$ and $V(f)=E(f)+1$; otherwise
in the case where $E(f)\geq 2$ and $V(f)=E(f)+1$, then the face $f$ can be shelled (corresponding to the reverse of the transition in Fig.~\ref{fig:step_canonical_3conn}(b)).
At each step, for a current active outer vertex $v$,   
let $N(v)$ be the number of separating faces incident to $v$.
Note that $v$ is admissible iff $N(v)=0$ and $v\notin\Cin$ 
(in which case $v$ can be shelled, corresponding to the reverse of the transition from Fig.~\ref{fig:step_canonical_3conn}(a)).  
By maintaining the quantities $E(f)$ and $V(f)$ for all internal faces touching the outer face, one can also maintain
the quantities $N(v)$ for all outer vertices (as well as their status: active or non-active). 
The shelling is done using two stacks $S$, $S'$: in $S$ are stored the current admissible vertices, 
in $S'$ are stored the current faces $f$ for which $E(f)\geq 2$ and $V(f)=E(f)+1$. At each step, at least one of the two stacks is non-empty (as we have proved above);
so it suffices to shell the top-vertex from $S$ or the top-face from $S'$, and then update the quantities $E(f)$, $V(f)$, $N(v)$.
Maintaining all these informations over the shelling procedure takes amortized time $O(|E|)$ (with $|E|$ the number of edges of the cylindric map), 
which is also $O(n)$ since $|E|\leq 3n$. More details on implementing such a procedure in linear time are given by Kant~\cite{Kan96} (for the shelling procedure
in the planar case). 

\vspace{.4cm}

\noindent{\bf Underlying forest and dual forest.} 
Given an internally $3$-connected cylindric map $G$ with no chord at $\Cin$, and endowed with a canonical ordering $\pi$, we define  
the \emph{underlying forest} $F$ for $\pi$ as the oriented subgraph of $G$ where each active 
vertex $v\in \Cext$  has outdegree $0$, and where each other vertex 
 has exactly one outgoing edge, which is
connected to the adjacent vertex $u$ of $v$ of largest rank in $\pi$. 
%
%
Since the edges are oriented in increasing labels, $F$ is a spanning (oriented) forest; each component of the forest is rooted at each of the active vertices on $\Cext$. 
The \emph{augmented map}  $\widehat{G}$ (seen as a map on the sphere) 
is defined as the map obtained from $G$ by adding a vertex 
$w_1$ inside $\Bin$, a vertex $w_2$ inside $\Bext$, and connecting 
 all vertices around $\Bin$ to $w_1$ and all active vertices around $\Bext$ to $w_2$.
 We define $\widehat{F}$ as $F$ 
plus all edges (of $\widehat{G}$) incident to $w_1$ and all edges incident to $w_2$.  
And we define the \emph{dual forest} $F^*$ 
for $\pi$ as the graph formed by the vertices of $\widehat{G}^*$ (the dual of  $\widehat{G}$)  and by the edges of $\widehat{G}^*$ 
that are dual to edges not in $\widehat{F}$. Each of the trees (connected components)
of $F^*$ is rooted at a vertex ``in front of'' each edge of $\Bin$, and the  
edges of the tree can be oriented toward this root-vertex 
(see Fig.~\ref{fig:canonical_3conn_complete} bottom right).
Similarly as in the triangulated case, for $e^*\in F^*$, we call \emph{path from $e^*$ to the root}
the path from $e^*$ to the root of the tree-component of $F^*$ to which $e^*$ belongs. 



\vspace{.4cm}

\subsection{Periodic drawing algorithm for internally $3$-connected cylindric maps with no active vertex for $\Cin$.} 
Given an  internally $3$-connected cylindric map $G$ with no active vertex for $\Cin$,  
 we first  compute 
a canonical ordering of $G$, and then draw $G$ in an incremental way, similarly as in the triangulated case.
A first useful remark is that, at any step $k$, if we look at a path $P$ of edges on $C_k$ connecting two consecutive active vertices for $C_k$ 
(i.e., $P$ starts at an active vertex for $C_k$, ends at an active vertex for $C_k$, and all internal vertices of $P$ are non-active),
then $P$ contains exactly one edge not in the underlying forest; this edge is called the \emph{bottom-edge} of $P$. 
We start with a cylinder of width $2|\Cin|$ and height $0$ (i.e.,
a circle of length  $2|\Cin|$) and draw the vertices of $\Cin$ equally spaced 
on the circle: space $2$ between two consecutive vertices (as in the triangulated case, it is possible to start with any configuration of points on a circle
such that any two consecutive vertices are at even distance).  
Then the strategy for each $k\geq 1$ is to compute the drawing of $G_k$ out of the drawing of $G_{k-1}$. The difference with the triangulated case
is that there are now two cases (those shown in Fig.~\ref{fig:step_canonical_3conn}).   

\vspace{.2cm}

\emph{(a) Addition of one vertex and several internal faces.} 
Consider the case of Fig.~\ref{fig:step_canonical_3conn}(a); let $v$ be the new added vertex, i.e., the unique vertex of $G_k$ not in $G_{k-1}$. 
Let $u_1,\ldots,u_s$ (with $s\geq 2$) be the neighbours of $v$ on $C_{k-1}$, 
such that the path $\gamma$ of $C_{k-1}$ from $u_1$ to $u_s$ has the outer face on its left. 
Let $e_1$ be the first edge on $\gamma$ and let $e_2$ be the last edge on $\gamma$.
There are two subcases. (1) If (in the drawing of $G_{k-1}$ obtained so far), $\mathrm{slope}(e_1)<1$ and $\mathrm{slope}(e_2)>-1$,
then, as in the triangulated case, we place $v$ at the intersection of the ray of slope $1$ starting from $u_1$ and the ray of slope $-1$ starting from $u_s$;
and we draw all the edges from $v$ to $u_1,\ldots,u_s$ as segments. 
 (2) If $\mathrm{slope}(e_1)=1$ or $\mathrm{slope}(e_2)=-1$, 
let $e_{\ell}$ be the bottom-edge of the part of $\gamma$ between $u_1$ and $u_2$ and let $e_r$ be the bottom-edge of the part of $\gamma$ between $u_{s-1}$ and $u_s$.  
Let $P_{\ell}$ (resp. $P_r$) be the path
in $F^*$ from $e_{\ell}^*$ (resp. $e_r^*$) to the root. Similarly as in the triangulated case (see Fig.~\ref{fig:one_step}),  
we stretch the cylinder by inserting
 a vertical strip of length $1$ along $P_{\ell}$ and another one along $P_r$.  
 After this, 
we insert (as in subcase (1)) the vertex $v$ at the intersection of the ray of slope $1$ starting from $u_1$ 
and the ray of slope $-1$ starting from $u_s$, and we connect $v$ to all vertices
 $u_1,\ldots,u_s$ by segments.  
The two rays from $u_1$ and $u_s$ actually intersects at a grid point since the Manhattan
distance between any two vertices on $C_{k-1}$ is even. 

\vspace{.2cm}

\emph{(b) Addition of one internal face.} 
Consider the case of Fig.~\ref{fig:step_canonical_3conn}(b), where we denote
 $u,v_1,\ldots,v_s,w$ ($s\geq 1$) 
the vertices of the new added chain, such that the path $\gamma$ on $C_{k-1}$
from $u$ to $w$ has the outer face on its left. Let $e_1$ be the first edge on $\gamma$ and $e_2$ the last edge on $\gamma$ (note that $e_1$ might be equal to $e_2$). 
Let $e$ be the bottom-edge of $\gamma$, and let $P$ be the path
in $F^*$ from $e^*$ to the root.  
There are two subcases. (1) If (in the drawing of $G_{k-1}$ obtained so far), $\mathrm{slope}(e_1)<1$ and $\mathrm{slope}(e_2)>-1$,
we insert a vertical strip of width $2s-2$ along $P$, increasing by $2s-2$ the $x$-span of each edge of $G_{k-1}$ dual to an edge in $P$. 
(2)  If $\mathrm{slope}(e_1)=1$ or $\mathrm{slope}(e_2)=-1$,
we insert a vertical strip of width $2s$ along $P$, increasing by $2s$ the $x$-span of each edge of $G_{k-1}$ dual to an edge in $P$.
Then we insert the vertices $v_1,\ldots,v_s$ into the drawing as follows. Let $R_u$ be the ray of slope $+1$ from $u$
and let $R_w$ be the ray of slope $-1$ from $w$. Let $q$ be the intersecting point of the two rays, denote by $y(q)$ its ordinate; 
let $S$ be the horizontal segment connecting $R_u$ to $R_w$ at ordinate $y(q)-s+1$. Note that $S$ has length $2s-2$. Then we insert $v_1,\ldots,v_s$
equally spaced (space $2$ between two consecutive vertices) on $S$, with $v_1$ at the left extremity and $v_s$ at the right extremity of $S$.
And we draw the edges of the chain $u,v_1,\ldots,v_s,w$ as segments. It is a simple exercise to show that the added chain does not cross the drawing of $G_{k-1}$.



\begin{figure}[tp]
\hspace*{-4mm}
\includegraphics[width=13.18cm]{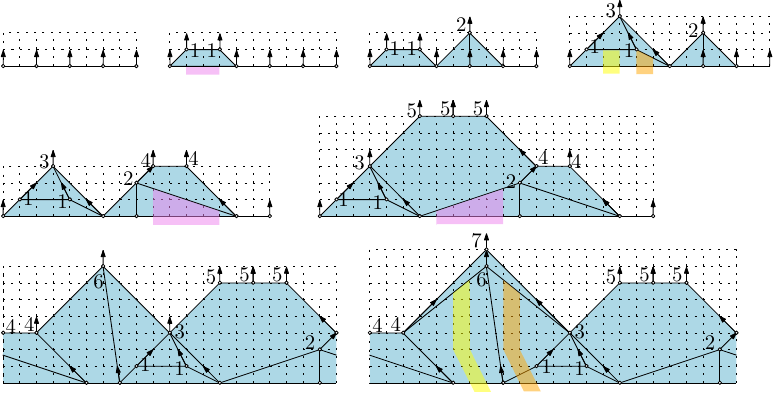}
 \caption{Complete execution of the algorithm computing an $x$-periodic drawing of an internally $3$-connected cylindric map 
(no active vertex for $\Cin$). The steps follow the canonical ordering (in increasing order) computed
in Fig.~\ref{fig:canonical_3conn_complete}.}
\label{fig:DrawingOnCylinder3conn}
\end{figure}

%

Fig.~\ref{fig:DrawingOnCylinder3conn}  
shows the execution of the algorithm on the example of Fig.~\ref{fig:canonical_3conn_complete}. 
The fact that the whole drawing of $G_k$ remains crossing-free convex relies on the  following inductive property, which is easily shown
to be maintained at each step $k$ from $1$ to $n$:

\vspace{.2cm}
 
\noindent\hspace{1cm}\begin{minipage}{11cm}
{\bf Pl}: 
In the upper boundary part $\gamma$ between two consecutive active vertices ---written (from left to right) as $\gamma=P_1,e,P_2$ with $e$
the bottom-edge of $\gamma$--- the edges of $P_1$ have slope $-1$, $e$ has slope in $\{-1,0,1\}$, 
and the edges of $P_2$ have slope $+1$.  

For each bottom-edge $e$ on $C_k$, 
let $P_e$ be the path in $F^*$ from $e^*$ to the root, let $E_e$ be the set of edges dual to edges in $P_e$,  
and let $\delta_e$ be any nonnegative integer.  
Then the drawing remains planar when successively  increasing by $\delta_e$ the $x$-span of all edges of $E_e$, for all bottom-edges $e\in C_k$. 
\end{minipage}

\vspace{.2cm}

\begin{remark}\label{rk:slopes}
Clearly the resulting drawing has also the property that for each edge $e$ in (resp. not in) the underlying forest $F$, the absolute value of the slope of $e$ is at least $1$ (resp. at most $1$). 
\end{remark}

We now prove the bounds on the grid-size ($w$ is the width and $h$ the height of the cylinder on which $G$ is drawn).  
If $|\Cin|=t$  then the initial cylinder is $2t\times 0$; and at each vertex insertion, the grid-width grows by $0$ or $2$.
Hence $w\leq 2n$. In addition, due to the slope conditions, the $y$-span (vertical span) of every internal face $f$
is not larger than the current width  at the time when $f$ is inserted in the drawing. 
Hence, if we denote by $v$ the vertex of $\Cext$ that is closest (at face-distance $d$)
to $\Cin$, then the ordinate of $v$ is at most $d\cdot(2n)$.  
And due to the slope conditions, the vertical span of $\Cext$ is at most 
$w/2\leq n$.  Hence the grid-height is at most $n(2d+1)$. The linear-time
complexity is shown next.

\vspace{.4cm}

\noindent{\bf Linear-time implementation.}  
The implementation is completely similar to the one in the triangulated case. In a first pass 
one computes the $y$-coordinates of  vertices and the $x$-span $r_e$ 
of each edge $e\in G$ at the time $t=k$ when $e$ appears in $G_k$ (as well one gets to know which extremity of $e$ is the left-end vertex). 
Afterwards if $e\notin F$ and $e\notin \Cext$,  
 the $x$-span of $e$ might further increase due 
to the insertion of new vertices. Let $s_e$ be the total further increase of the $x$-span undergone by $e$ after its appearance.
Let $w_e\geq 0$ be the increase of the $x$-span undergone by $e$ at the step $k$ such that $e\in C_{k-1}$ and $e\notin C_k$ if such a step exists (otherwise $w_e=0$); 
$w_e$ is called the \emph{weight} of $e$
(the quantities $w_e$ can be computed in a first pass together with the quantities $r_e$).  
Let $T_e^*$ be the subtree hanging from $e^*$ (including $e^*$) in the dual forest $F^*$, and let $W_e$ be the sum of the weights of (the dual of) all edges in $T_e^*$. 
Then, as in the triangulated case, $s_e=W_e$.  Since the quantities $s_e$ are easily computed in linear time from the quantities $w_e$ (starting from the leaves and going
up to the roots of $F^*$), this gives a linear time implementation.

\vspace{.2cm}

To sum up, we have proved Proposition~\ref{prop:triconn_cyl_first} for internally $3$-connected cylindric maps with no active vertices for $\Cin$. 

\vspace{.2cm}

\begin{remark}\label{rk:stretch_3conn}
As in the triangulated case (Remark~\ref{rk:stretch_triang}), if we consider  the \emph{initial-stretch} $R=(r_e)_{e\in \Cin}$ and the \emph{final-stretch} $T=(t_e)_{e\in \Cin}$, 
then the vector $S:=T-R$ is an invariant (does not depend on $R$), because $s_e=t_e-r_e$ only depends on the canonical ordering.  
Hence if a certain drawing yields final-stretch vector $T$ (starting from 
initial-stretch vector $R$), 
then for any vector $T'$ of the form $T+2V$ ---with $V$ a vector of non-negative integers---
one can redraw the graph to have final-stretch vector $T'$ (by taking $R'=R+2V$ as initial-stretch vector instead of $R$). 
\end{remark}

\begin{remark}\label{rk:recover_Kant}
Our algorithm extends Kant's algorithm~\cite{Kan96}, which works for internally $3$-connected planar maps (i.e., planar maps with 
a simple outer face contour and where all 2-separating curves have to pass by the outer face).  
Indeed, if we are given an internally $3$-connected plane graph $G$ (with an outer root-edge) we can
turn it into a cylindric internally $3$-connected map $\tilde{G}$ by adding a vertex of degree $2$ connected to the two extremities of the root-edge
(we also declare as active all outer vertices of $\tilde{G}$ that are not one of the three vertices of the inner boundary). 
Then, Kant's drawing of $G$ is recovered from the periodic drawing
of $\tilde{G}$ upon deleting the added vertex of degree $2$, see Fig.~\ref{fig:Kant}. 
\end{remark}

\begin{figure}
\begin{center}
\includegraphics[width=12cm]{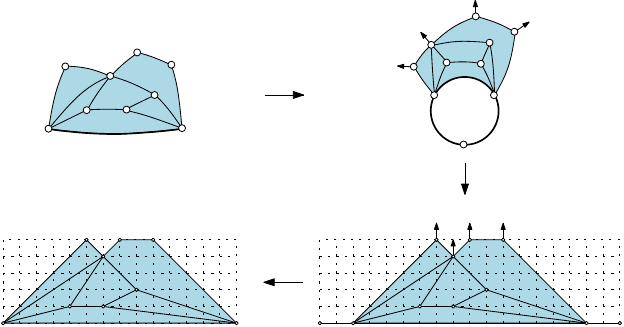}
\end{center}
\caption{Kant's algorithm for internally $3$-connected plane graphs is recovered from our algorithm by adding a vertex of degree $2$ to complete
the inner boundary.}  
\label{fig:Kant}
\end{figure}

\subsection{Allowing for non-contractible $2$-separating curves (no active vertex on $\Cin$).}\label{sec:allow_2sep_curves}
The method (canonical ordering and incremental drawing algorithm) is easily extended to essentially internally $3$-connected maps
with no $1$-separating curve. Additionally we require here that at least two outer vertices are active 
(the case of one outer vertex active will be treated in the next section). Let $G$ be such a cylindric map.
 The definition of canonical ordering
for $G$ is exactly the same as for internally 3-connected cylindric maps, 
adding the possibilities that $G_k$ is obtained from $G_{k-1}$
as shown in Fig.~\ref{fig:step3conn2}. 
\begin{figure}[t]
\begin{center}
\includegraphics[width=13cm]{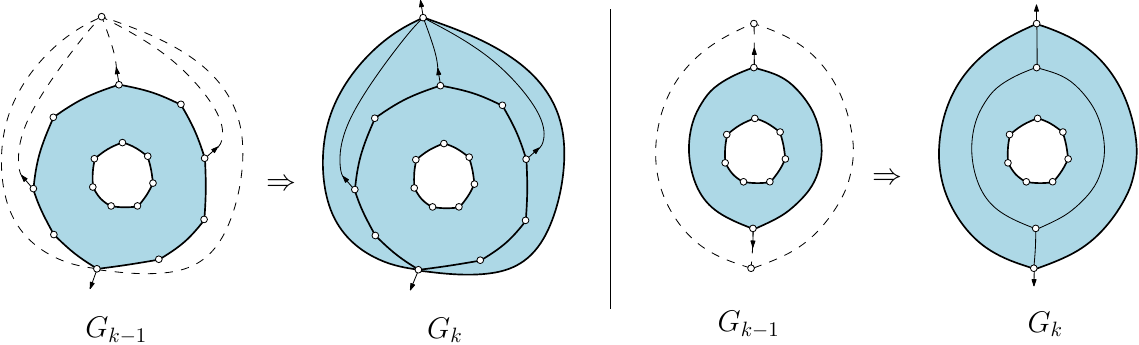}
\end{center}
\caption{When non-contractible 2-separating curves are allowed (and there are at least two active vertices for $\Cext$), the additional cases shown here might occur
for the transition from $G_{k-1}$ to $G_k$. 
} 
\label{fig:step3conn2}
\end{figure}

\begin{figure}
\begin{center}
\includegraphics[width=13cm]{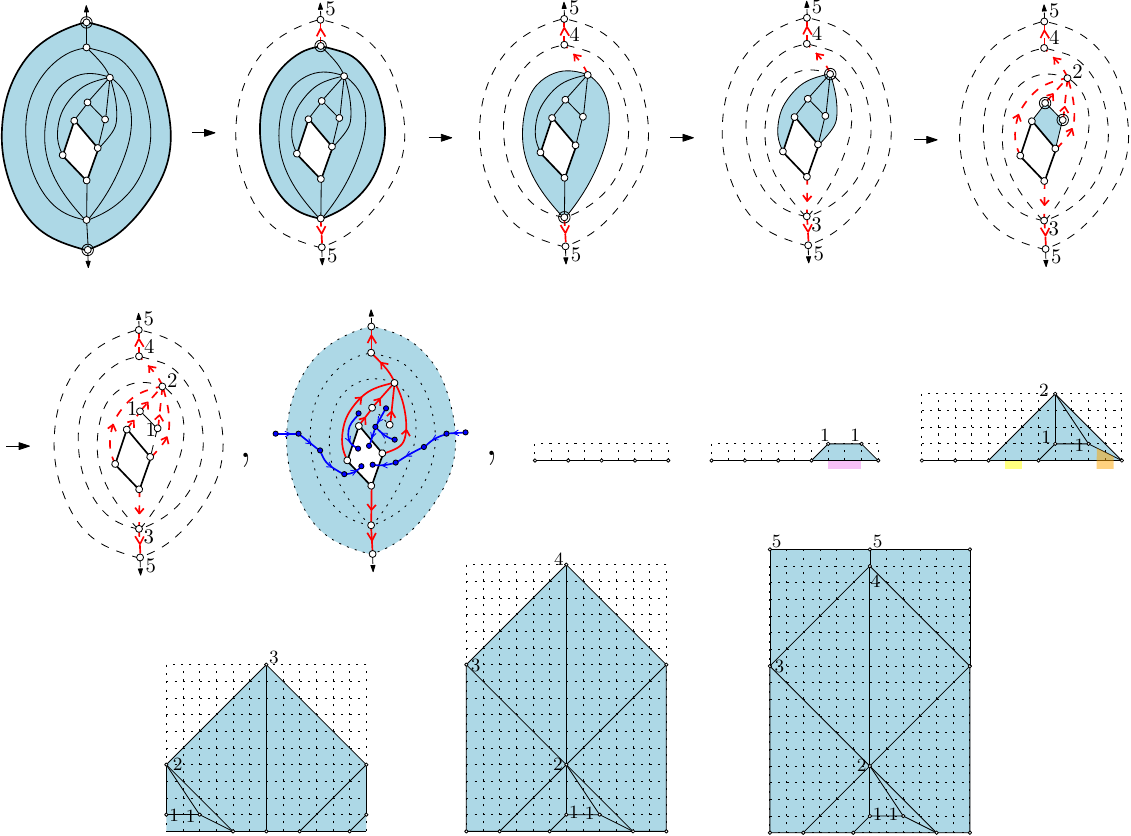}
\end{center}
\caption{(Before first comma): The shelling procedure for an essentially internally $3$-connected cylindric map $G$ with at least two active vertices for $\Cext$, 
with no active vertex for $\Cin$ nor 1-separating curve; (after first comma) the underlying forest and dual forest; (after second comma) the incremental drawing algorithm.} 
\label{fig:shelling2_3conn}
\end{figure}

Such a canonical ordering can be computed by a shelling procedure that extends the one of Section~\ref{sec:canonical_3conn}.
Call a $2$-separating curve \emph{internal} if it passes by two internal faces and does not pass by two outer vertices; in terms of the associated corner-map $S$ such a 2-separating curve corresponds
to a separating 4-cycle where the two incident white vertices (vertices of the cylindric map) are not
both on $\Cext$.  
This time, a vertex on the outer boundary and not on the inner boundary is called \emph{admissible} if it is not incident to a separating face  
nor incident to an internal 
$2$-separating curve. By the same arguments as in Section~\ref{sec:canonical_3conn}, one can show that if there is a separating face, then either there is 
an admissible vertex (which can be shelled) or there is an internal face (sharing a path of $r\geq 2$ edges  with the outer boundary) that can be shelled. 
If there is no separating face but there is at least one internal 2-separating curve, one can
check that, if two vertices of $\Cext$ are incident to an internal 2-separating curve,
then one must be in the situation in the right-part of Fig.~\ref{fig:step3conn2}. 
Otherwise at most one  vertex of $\Cext$ is incident to a 2-separating curve, 
and since there are at least two outer vertices active, at least one of them is admissible.   
Finally, in case there is no separating face nor 2-separating curve, then any active outer vertex is admissible
and can be shelled. 
One can continue shelling iteratively until there just remains $\Cin$ (since there is no 1-separating curve, it is easy to see that the 
property of having at least two active outer vertices is automatically maintained).  
A linear time implementation is also readily obtained by maintaining, for each outer vertex, 
how many separating faces and how many internal 2-separating curves it is incident to. 
Note that such a canonical ordering also induces an underlying forest $F$ and an underlying dual forest $F^*$. 
Finally, the incremental drawing algorithm (and linear implementation using the dual forest) 
works in the same way as for internally 3-connected cylindric maps. 
An example is shown in Fig.~\ref{fig:shelling2_3conn}. 
(the case in the right-part of Figure~\ref{fig:step3conn2} 
is new but easy to handle, it is shown in the last
step of the drawing in Fig.~\ref{fig:shelling2_3conn}).  
 
The grid bounds are also the same as for internally $3$-connected cylindric maps (the arguments to obtain the bounds in the simple case did not use the fact
that there are no 2-separating curves). So this gives Proposition~\ref{prop:triconn_cyl_first} for essentially internally $3$-connected cylindric maps with no 
1-separating curve and with at least two outer vertices active. 
Finally, Remark~\ref{rk:stretch_3conn} still holds here (the arguments are the same).

\vspace{.2cm}

\subsection{Allowing for 1-separating curves (no active vertex for $\Cin$).}
We finally explain how to deal with 1-separating curves. Similarly as in the triangulated case, we do not extend the notion of canonical ordering
but simply decompose (at 1-separating curves turned into loops) such a cylindric map into a ``tower'' of components separated by loops. 
Let $G$ be an essentially internally $3$-connected cylindric map with $n$ vertices. 
Note that if $G$ has a loop $e$, then there is a 1-separating curve ``just inside'' $e$ (if $e$ is not
$\Cin$) and there is a 1-separating curve ``just outside'' $e$ (if $e$ is not $\Cext$). 

Similarly as for cylindric triangulations, there are a few cases to treat:

\vspace{.2cm}

\emph{(a) $\Cin$ is a loop, the unique 1-separating curve is the one juste outside  $\Cin$, and there are at least two active vertices.} 
In that case, the algorithm of Section~\ref{sec:allow_2sep_curves} 
(canonical ordering, shelling procedure, 
and incremental drawing procedure) works in the same way, see Fig.~\ref{fig:draw_one_inn_loop_3conn} for an example.
In addition, by the arguments of Remark~\ref{rk:stretch_3conn}, 
for any $n'\geq n$, $G$ has a periodic drawing of width $2n'$ and height at most $n'(2d+1)$, with $d$ 
the face-distance between the two boundaries.

\begin{figure}
\begin{center}
\includegraphics[width=10cm]{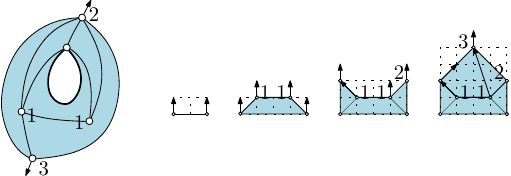}
\end{center}
\caption{Drawing algorithm when there is a unique 1-separating curve that 
is just outside the inner boundary (which is a loop).} 
\label{fig:draw_one_inn_loop_3conn}
\end{figure}

\vspace{.2cm}

\emph{(b) $\Cext$ is a loop, $\Cin$ is possibly a loop, and the only 1-separating
curves are just inside $\Cext$ and just outside $\Cin$.} 
In that case, there is the possibility that $\Cin$ and $\Cext$
are both loops and are connected by one edge $e$, in which
case the drawing is very easy (draw $\Cin$ and $\Cext$ as 
horizontal segments, and $e$ as a vertical segment). 
Otherwise, 
let $u$ be the unique outer vertex; it is incident to a 1-separating
curve; equivalently it is incident to a non-contractible 2-cycle $c'$  
in the corner-map $S$. Then, since we are not in the first 
case above, it is easy to see that, in $S$, 
there is a non-contractible 4-cycle ``directly inside" $c'$. 
Hence the set of non-contractible 4-cycles is non-empty. 
Let $c$ be the ``innermost" non-contractible 4-cycle,
which we now consider in $G$ in the form of a 2-separating curve.   
  Let $f_1,f_2$ the two internal faces visited by $c$ and let $v$
the vertex different from $u$ and visited by $c$. If there is an edge of $f_1$ connecting $u$ and $v$,
call $e_1$ this edge, otherwise draw an edge, again called $e_1$, inside $f_1$ that connects $u$ and $v$. Similarly if there is an edge of $f_2$ connecting $u$ and $v$,
call $e_2$ this edge, otherwise draw an edge $e_2$ inside $f_2$ that connects $u$ and $v$. 
Call $\tilde{c}$ the 2-cycle formed by $e_1$ and $e_2$. Cutting along $\tilde{c}$ 
we obtain two components: a plane 3-connected map $T$ of outer degree $3$,  and a cylindric
essentially internally 3-connected map $G'$ of outer degree $2$ and such that $G'$
has no internal 2-separating curve incident to $u$, and no 1-separating curve, except possibly
just outside $\Cin$ if $\Cin$ is a loop. Hence, 
if we declare both outer vertices of $G'$ as active, then $u$ is admissible and we can apply the 
results of Section~\ref{sec:allow_2sep_curves}: 
there is a canonical ordering of $G'$ such that $u$ is the first
shelled vertex; and this canonical ordering yields a periodic drawing of $G'$ where the 
upper boundary is made of $e_1$, of slope $-1$, and $e_2$, of slope $+1$. 
In addition, denoting by $2m$ the maximum of the widths of the drawings of $T$ and of $G'$,
one can redraw the drawing of smaller width so that both drawings have width $2m$.
Then  the drawing of $T$ (taken upside down) fits into the upper boundary of $G'$,
 yielding a periodic drawing of $G$ of width $2m$, see Fig.~\ref{fig:draw_one_ext_loop_3conn}.
In that case, the height is at most $2dm$ where $d$ is the face-distance between the two boundaries  (indeed, in the usual bound $h\leq (2d+1)m$
the $+1$ in the parenthesis is due to the vertical extension of the upper boundary, which is $0$ here). 
More generally, by the arguments of Remark~\ref{rk:stretch_3conn}, for any $n'\geq m$ there is a periodic drawing of $G$ with horizontal lower and upper boundaries,
of width $2n'$ and height at most $2dn'$. It remains to show that, if $e_1$ (similarly $e_2$)
was not present in $G$, then the deletion of $e_1$ keeps the drawing (weakly) convex.
To get convinced of it, one just has to notice that, if $e_1$ is not in $G$
and thus is to be deleted, then the properties of the drawings (Remark~\ref{rk:slopes}) guarantee that the ($\pi/2$) 
sector between $e_1$ and $e_2$ in the drawing of $T$ contains at least one edge (different from $e_1,e_2$); and in the drawing of $G'$ the $\pi/2$ sector around $v$ starting (in ccw direction)
from $e_1$ also contains at least one edge (different from $e_1$). Hence the angle at $v$ left by the deletion of $e_1$ is at most $\pi$. Similarly the angle at $u$ left by the deletion of $e_1$ is also at most $\pi$; and the angles at $u,v$ left by the deletion of $e_2$ (if $e_2$
is to be deleted) are at most $\pi$.

\begin{figure}
\begin{center}
\includegraphics[width=13cm]{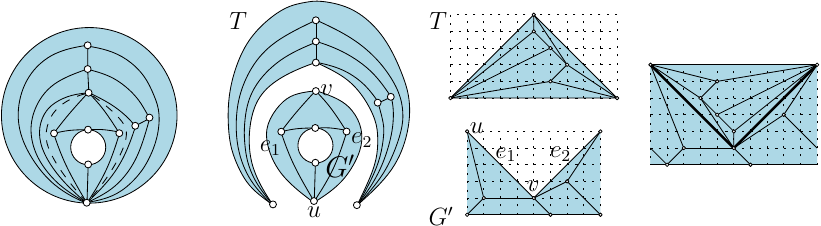}
\end{center}
\caption{Drawing algorithm when $\Cext$ is a loop and the unique 1-separating curve
is just inside $\Cext$: the map is split along the innermost 2-separating curve.} 
\label{fig:draw_one_ext_loop_3conn}
\end{figure}

\vspace{.2cm}

\emph{(c) General case (no active vertex at $\Cin$).} 
 For each 1-separating curve $c$, if $c$ is not just inside
nor just outside a loop, then we draw a loop-edge along $c$.  Also, if there is just one active vertex
and $\Cext$ is not a loop, then we add a loop at the unique active vertex, so that $\Cext$ becomes
a loop. 
After doing this, let $r$ be the number of loops (note that the case $r=0$ has already
been treated in the previous sections, so we can assume $r\geq 1$), and 
let $\ell_1,\ldots,\ell_r$ be the sequence of nested loops of $G$ (from innermost to outermost) and let $G^{(0)},\ldots,G^{(r)}$
be the $r+1$ components that result from cutting successively along all these loops, see Fig.~\ref{fig:stack_loops_3conn}.  
For $i\in [0..r]$ let $d_i$ be the face-distance between the two boundaries in $G^{(i)}$; and let $d$ be the face-distance between the 
two boundaries of $G$. Note that $d=\sum_id_i$.  
By construction, each component $G_i$ has its 1-separating curves either just inside the outer
boundary (if the outer boundary is a loop, which occurs if $i\in[0..r-1]$) or just outside the inner boundary (if the inner boundary
is a loop, which occurs if $i\in[1..r]$). 
Hence each component $G_i$ has a periodic drawing (according to cases (a) and (b)) 
whose boundaries are drawn as horizontal lines (except possibly for the outer boundary 
of $G_r$). 
Let $2n_1,\ldots,2n_r$ be the widths of the drawings of $G^{(1)},\ldots,G^{(r)}$ thus obtained. 
Let $m=\mathrm{max}(n_1,\ldots,n_r)$. By the arguments of Remark~\ref{rk:stretch_3conn}, each of the graphs $G^{(i)}$ can be redrawn so as to have width $2m$. 
Stacking up all these drawings, we obtain a periodic drawing of $G$ of width $2m$ and height at most $m(2d+1)$. 

\begin{figure}
\begin{center}
\includegraphics[width=12cm]{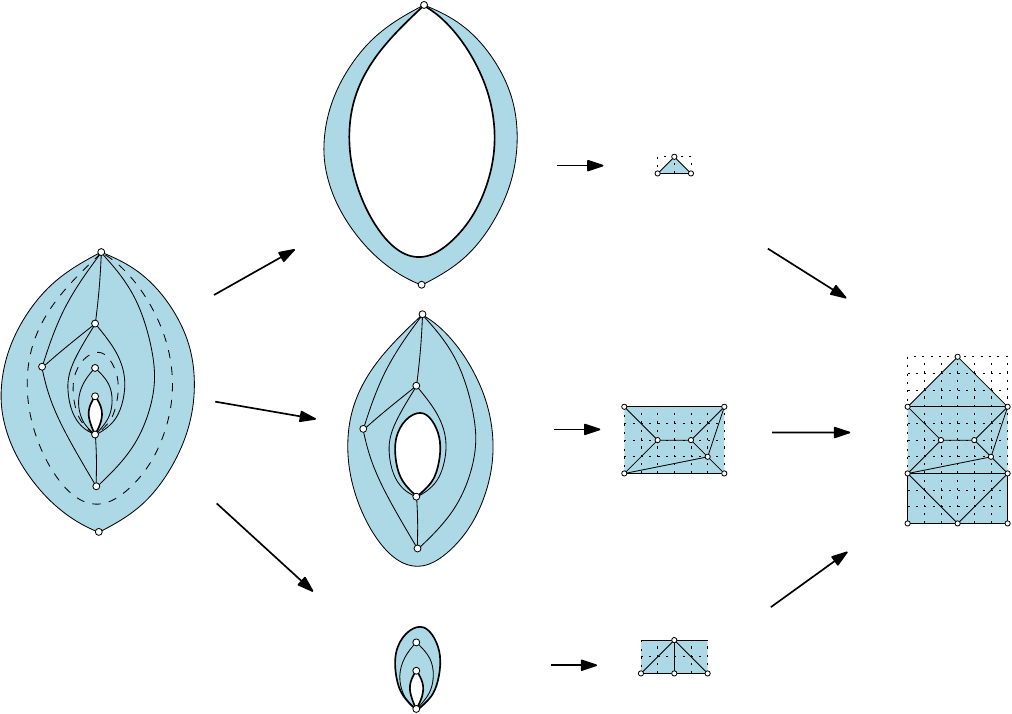}
\end{center}
\caption{Drawing an essentially internally 3-connected cylindric map $G$ (no active vertex for $\Cin$); at first loops (dashed in the left drawing) are drawn along 1-separating curves, then 
$G$ is  decomposed
at its loops; each component is redrawn so that the component-drawings  have 
the same width, and can be 
stacked up to obtain a periodic drawing of $G$.}
\label{fig:stack_loops_3conn}
\end{figure}

It remains to check that the deletion of the added loops keeps the drawing (weakly) convex.  
In the periodic drawing the loop is drawn as an horizontal segment $S=[a,b]$, where we denote by $a$ the left end and $b$ the right end. The properties of the drawings (Remark~\ref{rk:slopes}) 
easily imply
that there is at least one (non-loop)  edge in the $\pi/4$ sector around $a$ starting (in counterclockwise direction) 
at $S$, and there is at least one edge in the $3\pi/4$ sector around $a$ starting (in clockwise
direction) at $S$. Hence the angle at $a$ left by the deletion of $S$ is at most $\pi$.
Similarly the angle at $b$ left by the deletion of $S$ is at most $\pi$.

This establishes Proposition~\ref{prop:triang_no_chord}.

\begin{figure}[t]
\begin{center}
\hspace{-.2cm}\includegraphics[width=12.4cm]{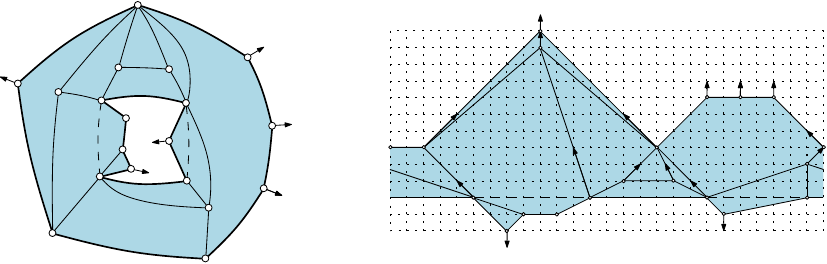}
\end{center}
 \caption{Drawing an internally 3-connected cylindric map with active vertices (and 2-chords) at $\Cin$.}
\label{fig:DrawingWithChords_3conn}
\end{figure}

\vspace{.2cm}

\subsection{\bf Allowing for active vertices at $\Cin$.}\label{sec:chords_cin_3conn} 
We finally explain how to draw an essentially internally 3-connected cylindric map $G$
in the general case, i.e., when allowing for active vertices at $\Cin$. The only 
requirement here is that there is at least one active vertex (for $\Cext$ or $\Cin$). 
We can assume that there 
is also at least one active vertex for $\Cext$ 
(otherwise the situation would be symmetric to the case where
 $\Cin$ has no active vertex, treated in the previous sections).

\subsubsection{The boundaries $\Cext$ and $\Cin$ are disjoint}  
Let $G$ be an essentially internally 3-connected cylindric map, with at least one active 
vertex for $\Cext$, and such that $\Cext$ and $\Cin$ are  disjoint (as in the triangulated
case, it is good here to imagine a standing cylinder, so that $\Cin$ is seen as the lower boundary
and $\Cext$ as the upper boundary).  As in the previous section,
we can consider the corner-map $S$ of $M$, and define \emph{a lower 2-chord} as 
a separating 2-chord for $\Cin$. A lower 2-chord is called \emph{maximal} if
it is maximal for the containment relation of enclosed cycles. 
For each maximal lower 2-chord, we can 
add an edge $e$ ``just above'' such a 2-chord (in the internal face it passes by) if $e$ is 
not already present.  
For each such edge $e$, let $Q_e$ be the so-called \emph{component under $e$}, which
is the subgraph of $G$ induced by $e$ and by the vertices and edges in the component below $e$.  Detaching all such components
we obtain a reduced cylindric map $G'$ such that, even when declaring all vertices of $\Cin$
as non-active, $G'$ is essentially internally 3-connected. We can thus draw $G'$ using the procedure of  Proposition~\ref{prop:triconn_cyl_first}. And we draw all detached components $Q_e$ using Kant's algorithm. 
 According to Remark~\ref{rk:stretch_3conn} we can then redraw $G'$ so that each edge $e$ has width large enough to plug
the drawing of $Q_e$ (recall that, if $|e|$ is the width of Kant's drawing of $Q_e$ and 
$w_e$ is the width of $e$ in the first obtained drawing of $G'$, then $G'$ can be redrawn
so that the width $\ell(e)$ of 
$e$ is the smallest integer at least $\mathrm{max}(|e|,w_e)$ and congruent to $w_e$ modulo $2$). 
Then for each such edge $e$, one can shift the left-extremity of the drawing of $Q_e$
so that the width of the bottom-edge of $Q_e$ is $\ell(e)$, and then (after rotation
by $\pi$) plug the drawing
of $Q_e$ into the lower segment for $e$ in the drawing of $G'$. The obtained
drawing of $G$ is clearly planar, see Fig.~\ref{fig:DrawingWithChords_3conn} for an example.  

We also have to check that, after deleting the edges that have been added (one such edge just
 above each maximal lower 2-chord, when needed), the drawing remains convex at corners
in internal faces. Let $e$ be such a deleted edge; note that (before
being deleted), $e$ is horizontal in the drawing; let $u$ be its right extremity and $v$ its left extremity.  The properties of the drawings easily imply that the sector of angle $3\pi/4$ around 
$u$ 
starting (in counterclockwise direction) at $e$ contains an edge (indeed $u$ is either
the left extremal vertex of $A$ or otherwise it has one incident 
edge in the underlying forest for $A'$). Moreover, since $u$ is the left extremal vertex of $Q_e$,
the sector of angle $\pi/4$ around 
$u$ 
starting (in clockwise direction) at $e$ contains an edge. We conclude that the corner at $u$
left by the deletion of $e$ has angle at most $\pi$ in the drawing. Similarly the angle at $v$
left by the deletion of $e$ has angle at most $\pi$. 
And by similar 
arguments as for cylindric triangulations, the width of the grid is at most $2n$ (with $n$
the number of vertices of $G$) and the height is at most $2n(d+1)$, with $d$ the face-distance
between the two boundaries. 

\subsubsection{The boundaries $\Cext$ and $\Cin$ meet} 
We now treat the case where $\Cext$ and $\Cin$ intersect. Recall that in that case $G$ decomposes
along such intersections as a cyclic sequence of elementary blocks called \emph{portions}, where
each portion is delimited by two intersection-vertices that are said to be \emph{extremal} (for the considered portion). A portion $A$ is \emph{complemented} 
by adding an outer edge $e_{\mathrm{out}}$  connecting
the two extremal vertices, so as to obtain a cylindric map $\widetilde{A}$. 
Note that $A$, if non-trivial, has
at least one non-extremal vertex that is active, either at $\Cin$ or at $\Cext$; without loss
of generality (one can exchange the roles of the outer and inner boundary for $\widetilde{A}$) 
we can assume that $A$ has at least one active vertex for $\Cext$ 
We can thus draw $\widetilde{A}$ by a procedure similar to the one described in the previous
section (by detaching components under maximal lower 2-chords). 

\begin{figure}[t]
\begin{center}
\includegraphics[width=13cm]{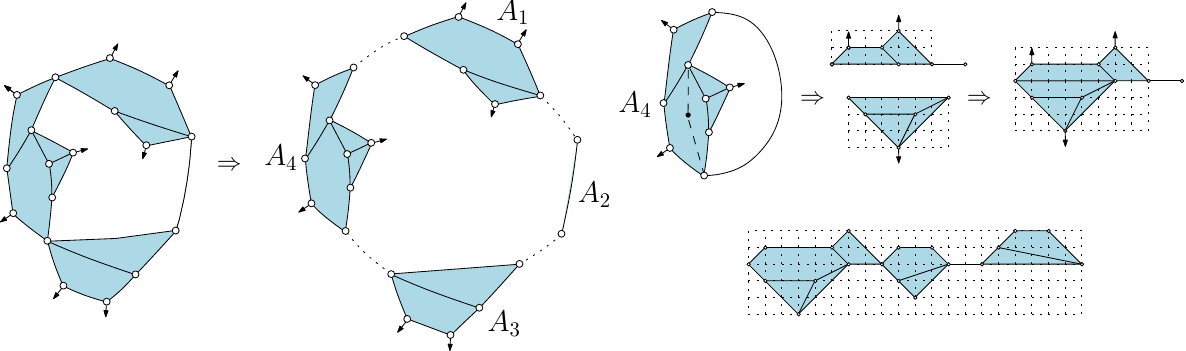}
\end{center}
 \caption{An internally essentially 3-connected cylindric map $G$ where $\Cext$ and $\Cin$ intersect
at $3$ vertices. Accordingly $G$ decomposes as a cylic sequence of $4$ portions (one of which is trivial). Each portion can be drawn (after adding an edge to complement it into a cylindric map)
by a decomposition at maximal lower 2-chords. The drawings of the portions are then
attached sequentially along the horizontal axis. 
} 
 \label{fig:portions}
\end{figure}
   
The overall drawing of $A$ is periodic planar, with $e_{\mathrm{out}}$ drawn as an horizontal
segment of width $2$.  
In addition, by similar arguments as in the previous section it can be checked that the drawing
is still convex after deleting the added edges (edges added just above maximal lower 2-chords). 

We can finally attach together (along the horizontal axis) the drawings of all portions thus obtained  (where the complementation edge is removed) 
to obtain an $x$-periodic drawing of $G$. By similar arguments as in the triangulated case
it can be checked that the width and height of the drawing are at most $2n$, with $n$ the 
number of vertices in $A$. This concludes the proof of Theorem~\ref{thm:triconn_cyl}.


\section{Periodic drawings on the torus}\label{sec:drawtorus}


\begin{figure}[t]
\begin{center}
\includegraphics[width=13cm]{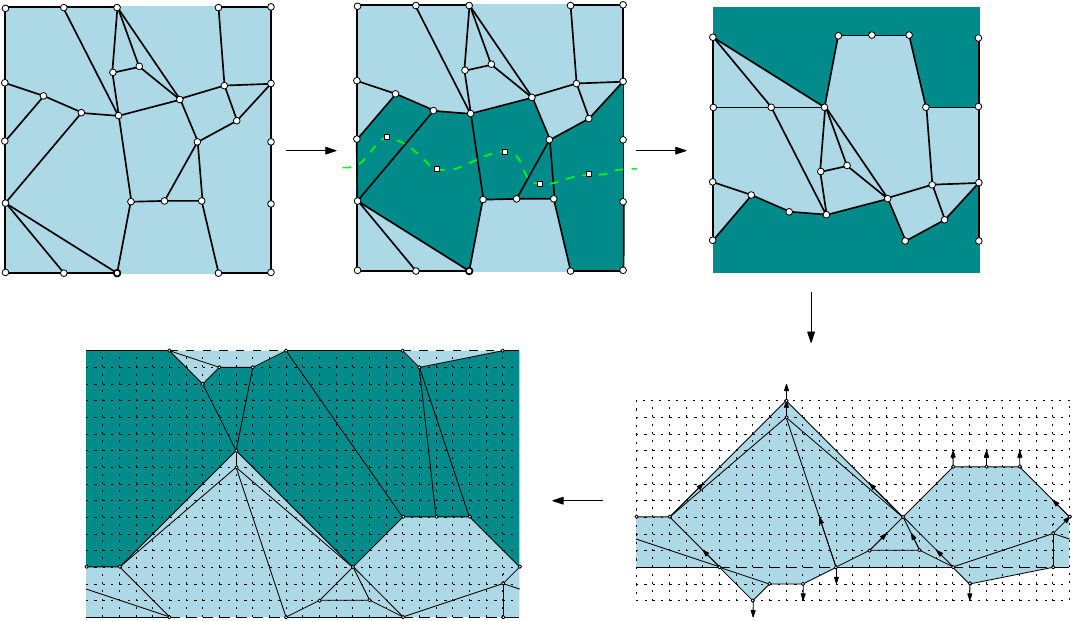}
\end{center}
 \caption{The successive steps to draw an (essentially) 3-connected toroidal map:
1) remove the edges inside a tambourine, (2) draw the obtained (essentially) internally 3-connected cylindric map, 3)
insert the edges of the tambourine back into the drawing.}
\label{fig:ToroidalDrawing}
\end{figure}


\subsection{Definitions and statement of the results}
A toroidal triangulation is a map on the torus with only triangular faces. 
A toroidal map is called \emph{simple} if it has no loop nor multiple edges, and is called \emph{essentially simple}
if its periodic representation in the plane is simple. 
A toroidal map is called \emph{$3$-connected} if it is $3$-connected as a graph, and is 
called \emph{essentially $3$-connected}
if its periodic representation in the plane is $3$-connected.
Let $G$ be a toroidal map. On the torus a non-conctractible curve is a closed curve that can not be continuously deformed to a point. 
Such a curve is called \emph{proper} if it meets $G$ only at vertices (not at edges); the \emph{length} of a proper non-contractible curve is 
the number of vertices it meets. The \emph{face-width} of $G$
is the minimum of the lengths of the proper non-contractible curves for $G$. 
A \emph{non-contractible cycle} of $G$ is a (simple) cycle of edges of $G$ that forms a non-contractible curve (the length of such a cycle is its number of edges).  
The \emph{edge-width} of $G$ is the minimum of the lengths of the non-contractible cycles of $G$. In this section we obtain the main result of the article:
\begin{theorem}\label{cor:ToroidalDrawing}
For each essentially $3$-connected toroidal map $G$, one can compute in linear time
a convex crossing-free straight-line drawing of $G$ on a periodic regular grid 
$\mZ/w\mZ\times\mZ/h\mZ$,  
where ---with $n$ the number of vertices and $c$ the face-width---  $w\leq 2n$ and $h\leq 1+2n(c+1)$. Since $c\leq\sqrt{2n}$, 
the grid area is $O(n^{5/2})$.  
\end{theorem}

The fact that $c\leq\sqrt{2n}$ follows from~\cite{Al78}. Indeed, an essentially simple toroidal graph $G$ (with $n$ the number of vertices)
can be triangulated, by adding edges but not adding vertices, so as to yield an essentially simple toroidal triangulation $\tilde{G}$. 
As shown in~\cite{Al78}, $\tilde{G}$ has a non-contractible cycle $\gamma$ of length at most $\sqrt{2n}$, and it is easy to see that this cycle
can be locally deformed into a proper non-contractible curve meeting $G$ at the vertices on $\gamma$.

\subsection{Tambourine: definition, existence, and computation}
Let $G$ be a toroidal map, and let $\Gamma_1,\Gamma_2$ be a pair of homotopic non-contractible cycles of $G$ that are oriented in the same direction.   
The pair $\Gamma_1,\Gamma_2$ is called a \emph{tambourine} if the area $A$ between $\Gamma_1$ and $\Gamma_2$ ---on the right of $\Gamma_1$ and on the left
of $\Gamma_2$--- is a ``ribbon of faces", i.e., $A$ is face-connected 
and the set of edges 
dual to edges in $A\backslash\{\Gamma_1,\Gamma_2\}$  
forms a non-contractible cycle homotopic to $\Gamma_1$ and $\Gamma_2$, see the
second drawing of Fig.~\ref{fig:ToroidalDrawing} for an example. 
(Note that $\Gamma_1$ and $\Gamma_2$ might  
share vertices and edges.) 
It is shown in the master's thesis of Arnaud Labourel (see also~\cite{Bon05}) 
that for each non-contractible cycle $\Gamma$ of a simple toroidal triangulation $G$, 
there exists a tambourine whose two cycles are homotopic to $\Gamma$.
We show here  that this holds more generally 
for essentially $3$-connected toroidal maps. 

\begin{lem}
Let $G$ be an essentially $3$-connected toroidal map and let $\Gamma$ be a non-contractible cycle of $G$.
Then there is a tambourine $\Gamma_1,\Gamma_2$ of $G$ whose two cycles are homotopic to $\Gamma$. 
Moreover one can compute such a tambourine in linear time.
\end{lem}
\begin{proof}
Cutting along $\Gamma$ we obtain a cylindric map $\widehat{G}$, which we consider 
in its annular representation. Let $C$ be the inner boundary of $\widehat{G}$.  
Let $\Gamma_1$ be the smallest (in terms of the enclosed area) cycle that strictly encloses $C$
(i.e., encloses $C$ and is vertex-disjoint from $C$). 
Then, let $\Gamma_2$ be the largest (in terms of the enclosed area) cycle that is strictly
enclosed in $\Gamma_1$ (i.e., is enclosed by $\Gamma_1$ and is vertex-disjoint from $\Gamma_1$). Orient $\Gamma_1$ and $\Gamma_2$ clockwise (in the annular representation). 
Let $A$ be the annular area between $\Gamma_1$ and $\Gamma_2$.
Note that, by minimality of the enclosed area, $\Gamma_1$ has no chord inside. Similarly, by maximality of the enclosed area,
 $\Gamma_2$ has no chord outside; hence both cycles delimiting $A$ have no chordal edge in $A$. 
It just remains to show that there is no vertex in the strict interior of $A$. Assume the set $E$ of such vertices is not empty, and let $H$
be a connected component of the subgraph induced by $E$.     
 Call \emph{vertex of attachment for $\Gamma_1$} (resp. for $\Gamma_2$) a vertex $w\in\Gamma_1$ (resp. $w\in\Gamma_2$) 
such that there is a path from a vertex of $H$ to $w$ that visits only vertices of $H$ before reaching $w$. 
By minimality of $\Gamma_1$, it is easy to see that there is a unique vertex of attachment $v_1$ for $\Gamma_1$.
Similarly, by maximality of $\Gamma_2$, there is a unique vertex of attachment $v_2$ for $\Gamma_2$. 
Let $H_2$ be the graph made by adding to $H$ the vertex $v_2$ as well as all edges connecting $v_2$ to vertices from $H$.
The annular representation is actually an embedding in the plane, 
so the faces of $H_2$ partition the plane. Hence the area corresponding to the inner boundary-face of $G$ is either
inside an inner face $f$ of $H_2$ or is inside the outer face of $H_2$. In the first case, the contour of $f$ yields a cycle 
 that is vertex-disjoint from $\Gamma_1$ and whose interior strictly contains the interior of $\Gamma_2$. 
This is impossible by maximality of $\Gamma_2$. 
In the second case ($\Bin$ in the outer face of $H_2$), 
let $H_{12}$ be the map obtained from $H_2$ by adding
the edges from $H$ to $v_1$. Two subcases can arise. If the area of $\Bin$ is in 
the outer face of $H_{12}$, then  there is a 2-separating curve (for $\widehat{G}$)  
passing by $v_1$ and $v_2$ and enclosing $H$ but not enclosing the inner boundary-face, 
which contradicts the fact that $G$ is essentially $3$-connected; 
if the area of $\Bin$ is in a bounded face of $H_{12}$ then there must exist a cycle $C$ passing
by $v_1$ and by vertices of $H$ such that $C$ encloses $\Bin$, 
which contradicts the minimality of $\Gamma_1$. 
In all cases we reach a contradiction, so $\Gamma_1,\Gamma_2$ form a tambourine.

Let us now justify that $\Gamma_1,\Gamma_2$ can be computed in linear time. Let $\widehat{G_1}$ be $\widehat{G}$ where vertices
and edges of $C$ have been deleted. 
Let $f$ be the inner face of $\widehat{G_1}$ whose interior contains the interior of $C$ (such a face is unique, even when $\widehat{G_1}$
has several connected components), and let $c$
be the unique simple cycle extracted from the contour of $f$ and such that the interior of $c$ contains the interior of $C$. 
Then $\Gamma_1$ equals $c$ (indeed $c$ is vertex-disjoint from $C$, the interior of $c$ contains the interior of $C$,
and no other cycle enclosed by $c$ can have these two properties). Now it is clear that computing $c$ (i.e., 
finding $f$ and then extracting $c$ from $f$) can be done in linear time.    
The computation of $\Gamma_2$ is very similar. Let $\widehat{G_2}$ be the subgraph of $\widehat{G}$ where all vertices and edges
on $\Gamma_1$ and exterior to $\Gamma_1$ have been deleted. Let $c'$ be the unique cycle extracted from the outer face contour of $\widehat{G_2}$
 and such that the interior of $c'$ contains the interior of $C$. Then $\Gamma_2$ equals $c'$, so $\Gamma_2$ is also easily computable in linear time.   
\end{proof}  

\subsection{The drawing algorithm}
Let $G$ be an essentially $3$-connected toroidal map with $n$ vertices, and let $\Gamma_1,\Gamma_2$ be a tambourine of $G$. 
 By deleting the edges that are strictly inside the tambourine, one 
  obtains a cylindric map 
$G'$ with $\Gamma_1$ the outer boundary and $\Gamma_2$ the inner boundary.
  Declare as active for $\Gamma_1$ each vertex $v\in\Gamma_1$ incident to at least an edge $e$ inside the tambourine (with an incidence on the right-side of $\Gamma_1$). 
And declare as active for $\Gamma_2$ each vertex $v\in\Gamma_2$ incident to at least an edge $e$ inside the tambourine (with an incidence on the left-side of $\Gamma_2$).
In the periodic representation of $G'$, if there is a 2-separating curve with no active vertex strictly inside, then it yields a 2-separating curve in the 
periodic representation of $G$, a contradiction. Hence $G'$ is essentially internally $3$-connected. 
We can apply the drawing algorithm of Theorem~\ref{thm:triconn_cyl} to obtain a convex periodic drawing of $G$ on an $x$-periodic grid $\mathbb{Z}/w\mathbb{Z}\times[0..h]$. 
If we augment the height $h$ of the drawing to $h'=h+w+1$, and 
then wrap the $x$-periodic grid $\mZ/w\mZ\times[0..h]$ into
a periodic grid $\mZ/w\mZ\times\mZ/h'\mZ$, and finally insert the
edges inside the tambourine as 
segments (we   
insert the edges in the tambourine $T$
in the unique way such that, looking from bottom to top, at least one edge in $T$ goes
strictly to the right, and all edges going strictly to the right have $x$-span at most $w$; in this way it is easy to check that the 
$x$-span of all edges in $T$ is at most $w$).  
Then the slope properties 
 ---edges on $\Gamma_1$ and $\Gamma_2$ have slope at most $1$ in absolute value
while edges inside the tambourine have slopes greater than $1$ in absolute value--- 
ensure that the resulting drawing is crossing-free and convex. Note that it is actually enough to augment the height
by the least value such that the edges of $T$ have slope greater than $1$ in absolute value.     
 See Fig.~\ref{fig:ToroidalDrawing} for an example,  
where the height is augmented by $4$ (whereas $w=26$).    

We now argue that we can find a tambourine $\Gamma_1,\Gamma_2$ so that the face-distance 
between the two boundaries $\Gamma_1$ and $\Gamma_2$ (in $G'$) is smaller than the 
face-width of $G$; and we can find it \emph{without having to compute} a curve $\Gamma_{\mathrm{min}}$ realizing the face-width
(this is crucial to obtain a linear-time complexity, since it is not known how to find such a curve $\Gamma_{\mathrm{min}}$ in linear time). 
Indeed, let $\{\Gamma_a,\Gamma_b\}$ be a basis of non-contractible cycles of $G$ (computable in linear time, using for instance a cut-graph). 
Then at least one of $\Gamma_a$ or $\Gamma_b$
is not homotopic (parallel) to $\Gamma_{\mathrm{min}}$. Let $\Gamma$ be one among $\{\Gamma_a,\Gamma_b\}$ that is not homotopic to $\Gamma_{\mathrm{min}}$, and let 
$\Gamma_1,\Gamma_2$ be a (computable) tambourine parallel to $\Gamma$.  
Since we are on the torus, $\Gamma_{\mathrm{min}}$ has to cross the tambourine $\Gamma_1,\Gamma_2$.  
 Hence, the 
distance between the boundary-cycles (after
deleting edges in the tambourine $\Gamma_1,\Gamma_2$) is smaller than the face-width. 
In other words, if we choose the one cycle among $\{\Gamma_a,\Gamma_b\}$ that yields
the smaller face-distance between the two boundaries of $G'$, then this distance $d$ 
is smaller than the face-width $c$ of $G$. The grid-size of the drawing of $G'$ satisfies $w\leq 2n$ and $h\leq 2n(d+1)$,
and the grid-size of the drawing of $G$ is $w,h'$ where $h'\leq h+w+1\leq 2n(d+1)+2n+1=1+2n(d+2)$. Since $d<c$ we conclude that
the grid-height of the drawing of $G$ is bounded by $1+2n(c+1)$. 

This concludes the proof of Theorem~\ref{cor:ToroidalDrawing}. 

\vspace{.2cm}

\noindent{\bf Acknowledgments.} 
 The authors 
thank N. Bonichon, D. Gon\c{c}alves, B. L\'ev\^eque, and B. Mohar for interesting discussions.



\small 
\bibliographystyle{abbrv}

\end{document}